\documentclass[11pt]{article}

\usepackage{amsmath, graphicx, latexsym, amssymb, amsthm, amsfonts}
\usepackage[mathscr]{eucal}
\usepackage{graphics}
\usepackage{color}
\usepackage{epsfig}
\usepackage{bbm}
\usepackage{hyperref}

\usepackage{tikz}

\newtheorem{theorem}{Theorem}
\newtheorem{assumption}{Assumption}

\newtheorem{lemma}{Lemma}

\theoremstyle{definition}
\newtheorem{definition}{Definition}
\theoremstyle{remark}
\newtheorem*{remark}{Remark}

\newcommand{\brac}[1]{\left( #1\right)}
\newcommand{\calig}[1]{ {\cal #1} }
\newcommand{\abs}[1]{\left| #1\right|}
\newcommand{\expect}[2][]{\mathbb{E}^{#1}\left[ {#2}\right]}
\newcommand{\prob}[1]{\mathbb{P}\left( #1\right)}
\newcommand{\closedbrac}[1]{\left[ #1\right]}
\newcommand{\norm}[1]{\left\lVert{#1}\right\rVert}
\newcommand{\indicator}[1]{ \mathbbm{1}_{\left\{#1\right\}} }

\newcommand{\MVD}[0]{\mathbb{D}_{\calig{M}}}
\newcommand{\MVDup}[0]{\mathbb{D}_{\calig{M}}^\uparrow}

\newcommand{\Dup}[0]{\mathbb{D}^\uparrow}

\newcommand{\CCup}[0]{\mathbb{C}^\uparrow}
\newcommand{\MVCnoatoms}[0]{\mathbb{C}_{\calig{M}_\sim}}
\newcommand{\MVCupnoatoms}[0]{\mathbb{C}_{\calig{M}_\sim}^\uparrow}
\newcommand{\MVDK}[0]{\mathbb{D}_{\calig{M}}^K}
\newcommand{\MVDupK}[0]{\mathbb{D}_{\calig{M}}^{\uparrow K}}

\newcommand{\MVCupK}[0]{\mathbb{C}_{\calig{M}}^{\uparrow K}}
\newcommand{\DK}[0]{\mathbb{D}_{+}^K}
\newcommand{\DupK}[0]{\mathbb{D}^{\uparrow K}}

\newcommand{\MVCupnoatomsK}[0]{\mathbb{C}_{\calig{M}_\sim}^{\uparrow K}}

\newcommand{\la}{\lambda}
\newcommand{\eps}{\varepsilon}
\newcommand{\ph}{\varphi}

\newcommand{\al}{\alpha}

\newcommand{\sig}{\sigma}
\newcommand{\del}{\delta}

\newcommand{\Om}{\mathnormal{\Omega}}

\newcommand{\N}{{\mathbb N}}

\newcommand{\R}{{\mathbb R}}

\newcommand{\E}{{\mathbb E}}

\newcommand{\PP}{{\mathbb P}}

\newcommand{\calB}{{\cal B}}

\newcommand{\calF}{{\cal F}}

\newcommand{\calL}{{\cal L}}
\newcommand{\calM}{{\cal M}}

\newcommand{\calS}{{\cal S}}

\newcommand{\calX}{{\cal X}}

\newcommand{\skp}{\vspace{\baselineskip}}
\newcommand{\supp}{{\rm supp}}

\newcommand{\lt}{\left}
\newcommand{\rt}{\right}

\newcommand{\To}{\Rightarrow}

\newcommand{\iy}{\infty}
\newcommand{\noi}{\noindent}

\begin{document}

\title{
Fluid limits for earliest-deadline-first networks
}
\author{Rami Atar\thanks{Viterbi Faculty of Electrical Engineering,
Technion -- Israel Institute of Technology, Haifa 32000, Israel}
\and Yonatan Shadmi${}^*$}
\maketitle
\begin{abstract}
This paper analyzes fluid scale asymptotics
of two models of generalized Jackson networks
employing the earliest deadline first (EDF) policy.
One applies the `soft' EDF policy,
where deadlines are used to determine priority but jobs do not renege,
and the other implements `hard' EDF, where jobs renege
when deadlines expire,
and deadlines are postponed with each migration to a new station.
The arrival rates, deadline distribution and service capacity
are allowed to fluctuate over time at the fluid scale.
Earlier work on EDF network fluid limits,
used as a tool to obtain stability of these networks,
addressed only the soft version of the policy,
and moreover did not contain a full fluid limit result.
In this paper, tools that extend the notion
of the measure-valued Skorokhod map are developed
and used to establish for the first time fluid limits
for both the soft and hard EDF network models.

\skp

\noi{\bf AMS subject classifications:} 60K25, 60G57, 68M20

\skp

\noi{\bf Keywords:}
measure-valued Skorokhod map,
measure-valued processes,
fluid limits, earliest deadline first, generalized Jackson networks

\end{abstract}


\section{Introduction}

This work continues a line of research initiated in \cite{atar2018}
and expanded in \cite{atar2018law},
in which a Skorokhod map (SM) is used in conjunction with
stochastic evolution equations in measure space, to characterize scaling limits
of queueing systems.
Specifically, the measure-valued Skorokhod map (MVSM) of \cite{atar2018},
an infinite-dimensional
analogue of various well-known SMs in the finite-dimensional
orthant, is a tool for analyzing policies that prioritize according
to a continuous parameter.
This notion has yielded new results on {\it fluid}
or {\it law of large numbers} (LLN)
scaling limits of queueing systems implementing
the policies {\it earliest deadline first} (EDF), {\it shortest job first} (SJF)
and {\it shortest remaining processing time} (SRPT) \cite{atar2018},
as well as EDF in a many-server scaling \cite{atar2018law}.
In the case of EDF, the state of the system is given by a measure $\xi$ on
$\R$, where $\xi(B)$ expresses the queue length associated with all
jobs in the buffer that have deadlines in the set $B$,
for $B\subset\R$ a Borel set.
The aforementioned continuous parameter corresponds to a job's deadline,
and the MVSM encodes priority by enforcing the rule that,
for every $x\in\R$, work with deadline $>x$ can be transferred
from the buffer to the server
only at times when $\xi((-\iy,x])=0$,
that is, when there is no work in the buffer associated with deadline $\le x$.
Thus the MVSM of \cite{atar2018} is useful for the analysis of a {\it single node}.
The goal of this paper is to address EDF in {\it network setting}.
To this end two multidimensional counterparts of the MVSM are developed
involving multiple measure-valued processes,
one consisting of a map that combines
properties of the MVSM and of the {\it
SM with oblique reflection in the orthant} (SMOR),
and another is where the MVSM serves as a building block
in a system of equations.

EDF is a rational choice of service policy for systems in which job urgency is quantifiable.
Its practical significance stems from its use in real time applications such as
telecommunication networks \cite{Aras1994},
real time operating systems \cite{but11},
and emergency department operations \cite{Hu2017}.
The policy has two main versions:
\begin{itemize}
    \item soft EDF: all jobs are served whether their deadline expires or not;
    deadlines serve only to determine the priority of a job,
    \item hard EDF: deadlines determine priority and
    jobs whose deadlines expire renege the system.
\end{itemize}
The soft and hard versions of this policy are also known in the literature
as EDF without and with reneging, respectively,
and EDF is also referred to as {\it earliest due date first served}.
A further subdivision is according to whether service is preemptive or non-preemptive; in this paper we restrict our attention to the non-preemptive
disciplines.
Although fluid limits of soft EDF queueing networks have been studied
before in relation to the question of their stability in \cite{Bramson2001},
\cite{Kruk2008network},
these papers have established only partial results as far as convergence was concerned
(most significantly, for stability analysis
only convergence along subsequences is required, and indeed
only such convergence was proved; see more details below).
In this paper, the two aforementioned counterparts of the MVSM yield
for the first time fluid limits for both the soft and the hard versions of EDF networks.

There has been great interest in this policy in theoretical studies.
It has been argued
to possess optimality properties, and much effort has been devoted to studying
its scaling limits \cite{Lehoczky1996}, \cite{Doytchinov2001},
\cite{decreusefond2008}, \cite{kruk2008}, \cite{Kruk2010}, \cite{kruk2011},
\cite{Atar2013}, \cite{Atar2014}, mostly in the single node case.
The hard EDF policy was shown to be optimal
with respect to expected number of reneging jobs \cite{Panwar1988},
for the $M/G/1+G$ and $G/D/1+G$ queues.
It was also showed that if there exists an optimal policy for
the $G/G/1+G$ queue then EDF is optimal.
Optimality properties of EDF were further studied
with respect to cumulative reneged work \cite{kruk2011}
as well as steady state probability of loss \cite{Moyal2013}.
Fluid models of multi-class soft EDF Jackson networks
were studied in \cite{Kruk2010}, focusing on invariant states of the dynamical
system and characterizing its invariant manifold.

As for scaling limits,
most effort has been devoted to the heavy traffic regime.
The paper \cite{Lehoczky1996} studied diffusion limits of soft EDF queues under heavy traffic assumptions as Markov processes in Banach spaces.
An expression for the lead time profile given the queue length was obtained.
An important line of research started from \cite{Doytchinov2001},
a paper that pioneered (along with \cite{grom-puha-will})
the use of measure-valued processes for scaling limits
of queueing-related models.
In this paper the diffusion limit of soft EDF queues was characterized.
Further significant development was in \cite{kruk2011},
that used SM techniques to show convergence of a hard EDF queue to a diffusion limit.
Specifically, the total workload process was shown to converge
to a doubly reflected Brownian motion,
and the limit of the underlying state process,
that is again a measure-valued process,
is given in terms of this reflected Brownian motion.
The fact that the total workload in the system is sufficient
to recover the complete state of the system is due to
a certain {\it state space collapse} (SSC) phenomenon,
which remarkably generalizes the well known SSC for priority queues.
SSC, as a phenomenon and as a tool for proving convergence,
is unique to the heavy traffic diffusion regime,
and thus cannot be of help in the fluid limit regime
that is of interest in this paper.

Hard EDF fluid limits were established in
\cite{decreusefond2008} and \cite{Atar2013}.
These papers considered the $M/M/1+G$ and, respectively, $G/G/1+G$ models.
The paper \cite{atar2018} introduced the MVSM and used it to establish
fluid limits for several models, including soft and hard EDF.
The paper \cite{atar2018law} employed this approach to analyze EDF
in a many-server fluid limit regime.

Scaling limits of EDF in network setting were studied in \cite{Lehoczky1997}
and \cite{kruk2004Netwrks} (in heavy traffic) and in
\cite{Bramson2001} and \cite{Kruk2008network} (at fluid scale).
The paper \cite{Lehoczky1997} extended \cite{Lehoczky1996}
to a network setting. Given the queue lengths at the nodes,
the Fourier transform of the lead time distribution
was represented as the solution of a fixed point equation.
The paper \cite{kruk2004Netwrks} extended \cite{Doytchinov2001}
to multi-class acyclic networks.

Much closer to the subject of this paper are the results on fluid scale limits
established in \cite{Bramson2001} and \cite{Kruk2008network}.
It is well known that fluid models are useful
in proving stability of queueing networks, by reducing the problem of
stability of the stochastic dynamical system representing the queueing network
into the stability of all solutions to a deterministic dynamical system,
usually represented in terms of deterministic evolution
equations, comprising the fluid model. This approach, that was
first formulated in general terms in \cite{dai95}, was taken in \cite{Bramson2001}
to establish the stability of soft EDF networks, in fact in the broader, multi-class setting,
provided that they are sub-critically loaded. The approach is based on showing stability
properties of {\it any} solution to the fluid model equations, and does not require
that uniqueness holds for these solutions (for a given initial condition).
As a consequence, uniqueness of solutions to the fluid model equations
is not required, nor is it established in \cite{Bramson2001}.
A related issue is that when using this method it suffices to establish convergence
of the rescaled processes along subsequences.
Indeed, fluid limits are proved only along subsequences, and so the
full convergence to a fluid limit is not established there.

It is also important to point out that the scaling in \cite{Bramson2001}
is such that the gap between  time of arrival and deadline vanishes
in the limit. Consequently, the asymptotics are indistinguishable from
that under a policy that prioritizes by order of arrival, namely {\it first in system first out}
(FISFO).
In contrast, under our scaling the gaps alluded to above
remain of order one, and are therefore captured by the limiting dynamics.
In this we follow the treatment in \cite{atar2018}
and \cite{atar2018law}. This aspect is also
similar to the nature of the limit dynamics
in most of the aforementioned papers on EDF in heavy traffic,
where the limiting system's
state is given by a nondegenerate measure that accounts for a variety
of deadlines.

The stability problem was also studied via the fluid limit approach in \cite{Kruk2008network}
in the more complicated case of preemptive multiclass EDF networks,
under the assumption that customers have fixed routes through the system.
In \cite{Kruk2008network} too it was not claimed, nor is it obvious, that the fluid model equations
uniquely characterize the fluid model solutions for a given initial condition,
and limits were only established along subsequences.
The same paper also studied the stability of hard EDF networks
(as well as several other network models with reneging),
though not via fluid models,
hence it did not address the problem of fluid limits of hard EDF networks.

Our first main result is the fluid limit for soft EDF Jackson networks.
The approach is based on a tool that we develop, that combines properties
of the aforementioned MVSM, which encodes the priority
structure, and SMOR, which encodes the structure of the
flow within the network.
It is well known since \cite{rei84} that the heavy traffic
asymptotics of Jackson networks is given by a reflected Brownian
motion in the orthant with oblique reflection vector field,
a process that can be represented in terms of a corresponding SMOR.
The same SMOR has been used since then to study Jackson networks
in other asymptotic regimes, namely fluid limits \cite[Ch.\ 7]{chen2001},
and large deviations \cite{atar1999large}.
The way we use it in this paper is as follows.
We represent the state of the system
in terms of a {\it vector measure} (an $\R^d$ valued set function)
$\xi$ on $\R$, where for a Borel set $B\subset\R$, $\xi(B)=(\xi^i(B))$
expresses the workload associated with jobs having deadlines
in $B$ in the various buffers $i$.
The resulting extended version of the MVSM, that we call the
{\it vector measure valued SM} (VMVSM)
is a transformation on the space of trajectories with values in
the space of vector measures on $\R$.
It has the property that for each deadline level $x\in\R_+$, the trajectory
$t\mapsto\xi_t([0,x])$ is given as the image of the data of the problem
(a suitable version of the so called netput process)
under a SMOR.
This provides a generalization of the concept from \cite{atar2018}.
The convergence result is valid under very mild probabilistic assumptions.
Note that fluid limits of soft EDF networks
immediately imply fluid limits of FISFO networks,
if the deadlines are set to be the arrival times.


Our second contribution is the treatment of hard EDF networks.
The model is considerably more difficult as far as scaling limits are concerned,
a fact reflected by the very small number of papers on the subject.
In particular, the method of \cite{atar2018} to handle hard EDF
for a single node does not extend to networks (in \cite{atar2018}
uniqueness for the fluid model solutions
is obtained via pathwise minimality, a property that
does not generalize to networks -- counterexamples can be constructed).
In our treatment we make two model assumptions that differ from the soft 
EDF model besides the obvious difference between the two policies.

First, the deadlines of a job are not fixed but increase whenever it
migrates to a new service station.
The problem appears to be considerably
harder with fixed deadlines, having to do with regularity
of the arrivals processes into each station.
That is, in a model with fixed deadlines, when one of the nodes
becomes supercritical, jobs begin to renege at some point,
causing a potentially large number of jobs to arrive at other stations
very close to their deadline. Consequently, small (asymptotically
negligible) perturbations in deadlines may result in large
(fluid scale) differences in the system's state.
We leave open the important question of establishing fluid limits with fixed deadlines.
Under the dynamic deadline assumption we are able to harness the
single node results from \cite{atar2018} to the network setting.
An important part of the proof is devoted to establishing regularity properties of
arrival processes into each station, composed of exogenous and endogenous arrival processes (due to routing within the system), affected
by the reneging in a nontrivial way.
The aforementioned assumption allows us to use an argument by induction over
squares in the time-deadline plane, where at each step the size of the square side is increased.
This is used in treating both the fluid model equations and
the convergence result.

Second, the service time distributions are assumed to
have bounded support.
This eases the notation for a certain technical reason
related to the possibility that, while a job is in service,
the (extended) deadline expires in the station it may be routed to.
This assumption is not difficult to remove
but we keep it because it does simplify the form of the model equations.


Finally we note that the techniques developed in this paper
are likely to be useful
for networks implementing other related disciplines, such as SJF.

The organization of this paper is as follows. At the end of this
section we introduce notation used throughout.
In \S\ref{sec2} we describe the VMVSM and the soft EDF queueing network,
and then state and prove the convergence result.
In \S\ref{sec3} we first describe the fluid model equations
and prove uniqueness of solutions to these equations,
then we formulate the queueing network model under hard EDF,
state the main convergence result and provide its proof.

\subsection*{Notation}

In what follows, $\R_+=[0,\iy)$.
For a Polish space $\calS$,
$\mathbb{C}_\calS$ and $\mathbb{D}_\calig{S}$
denote the space of continuous and, respectively, c\'adl\'ag
functions
from $\R_+$ into $\calS$; if $\calig{S}=\mathbb{R}$ we simply write
$\mathbb{C}$ and $\mathbb{D}$.
Denote by $\mathbb{D}_+$, respectively $\mathbb{D}^{\uparrow}$,
the subset of $\mathbb{D}$ of non-negative functions, respectively of non-decreasing and non-negative functions, and apply a similar notation to $\mathbb{C}$.
We also define 
\begin{align*}
    \mathbb{D}_0^K=\left\{f\in\mathbb{D}^K:\quad f^i\brac{0}\geq 0\text{ for all }1\leq i\leq K\right\}.
\end{align*}
${\cal M}$ denotes the space of finite Borel measures on $\mathbb{R}_+$, endowed with the topology of weak convergence. Its subset of atomless measures is denoted by ${\cal M}_\sim$. It is well known that the topology of weak convergence is metrized by the Levy-Prohorov metric, denoted $d_\calig{L}$. 
Define the subset of $\mathbb{D}_\calig{M}$ of non-decreasing elements as $\mathbb{D}_\calig{M}^\uparrow=\{\zeta\in\mathbb{D}_\calig{M}$: $t\mapsto\int_{\R_+} f(x)\zeta_t(dx)$ is non-decreasing for any continuous, bounded, non-negative function $f\}$. 
With a slight abuse of notation we denote
$\DupK=\brac{\Dup}^K$ and $\MVDupK=\brac{\MVDup}^K$.
The support of a measure $\zeta\in\calM$ is denoted by $\text{supp}\closedbrac{\zeta}$. For $\zeta\in\calM$ we write $\zeta[a,b]$ for $\zeta([a,b])$,
and similarly for $[a,b)$, etc. 
We use the convention that for $a>b$, $\closedbrac{a,b}=\emptyset$.
$\mathbb{D}$ and $\mathbb{D}_{\cal M}$ are equipped with the corresponding $J_1$ Skorokhod topologies, and $\mathbb{D}^K$ and $\mathbb{D}_{\cal M}^K$ are equipped with the product topologies. 
For $x\in\R^K$, $x^i$ denotes the $i$-th component,
and $\norm{x}=\max_{1\leq i\leq K}\abs{x^i}$.
For $x\in\mathbb{D}^K$ denote $\norm{x}_T=\sup_{t\in\closedbrac{0,T}}\norm{x\brac{t}}$.
The modulus of continuity of a function $f:\mathbb{R}_+\to\mathbb{R}$ is denoted by 
\begin{align*}
    w_T\brac{f,\epsilon}=\sup\{\abs{f\brac{s}-f\brac{t}}:s,t\in\closedbrac{0,T},\abs{s-t}\leq\epsilon\}.
\end{align*}

\section{Soft EDF networks}\label{sec2}

In this section we develop the VMVSM, which
combines elements of the MVSM and the SMOR. We then use it
to represent and establish the fluid limit of soft EDF networks.
We begin by introducing,
in \S\ref{sec211}, the MVSM, and then a Skorokhod problem
whose well posedness is stated and proved.
This gives rise to the VMVSM.
Then, in \S\ref{sec:soft prelimit}, the queueing model and its rescaling
are introduced. We then state the main result.
Its proof appears in \S\ref{sec: analysis & proofs}.

\subsection{Model and main results}

\subsubsection{A Skorohod problem in measure space}\label{sec211}

The measure valued Skorokhod problem (MVSP)
introduced in \cite{atar2018} is as follows.
\begin{definition}[MVSP]\label{def:MVSP}
Let $\brac{\alpha,\mu}\in\MVDup\times\Dup$. Then $\brac{\xi,\beta,\iota}\in\MVD\times\MVDup\times\Dup$ is said to solve the MVSP for the data $\brac{\alpha,\mu}$ if, for each $x\in\mathbb{R}_+$,
\begin{enumerate}
    \item $\xi\closedbrac{0,x}=\alpha\closedbrac{0,x}-\mu+\beta\brac{x,\infty}+\iota$,
    \item $\int_{\left[0,\infty\right)}\xi_s\closedbrac{0,x}d\beta_s\brac{x,\infty}=0$,
    \item $\int_{\left[0,\infty\right)}\xi_s\closedbrac{0,x}d\iota\brac{s}=0$,
    \item $\beta\left[0,\infty\right)+\iota=\mu$,
\end{enumerate}
where in 2. and 3. the integration variable is $s$.
\end{definition}
It was shown that there exists a unique solution to the MVSP (\cite[Proposition 2.8]{atar2018}), and thus, the MVSP defines a map
$(\al,\mu)\mapsto(\xi,\beta,\iota)$. Moreover, this map has certain continuity 
properties that were key in establishing scaling limits for both soft and hard EDF
\cite[Theorem 5.4]{atar2018}.
The approach that we take here builds on these ideas but replaces
$\xi$ and $\beta$ by vector measures.
We motivate our definitions by describing a fluid model for an EDF network. 

An informal description is followed by a precise mathematical formulation.
Consider a queueing network with $K$ nodes; each node consists of a server and a queue, that accommodates fluid. Let $P\in\R^{K\times K}$ be a given substochastic
matrix.
Exogenous arrivals form an input to the system.
They consist of fluid entering the various queues. Each unit of arriving mass
has an associated distribution of deadlines.
The servers prioritize the mass with the smallest deadline. 
After service, the mass splits between the nodes according to $P$.
That is, the stream leaving server $i$ splits so that a fraction
$P_{ij}$ routes to queue $j$.
These streams form the endogenous arrival processes.
The remaining fraction, $1-\sum_{j=1}^KP_{ij}$, exits the system.
The deadlines do not vary during the entire sojourn in the system.

Define for the $i$-th queue, respectively, the exogenous arrival process, cumulative potential effort, queue content, and departure process, as follows: 
\begin{alignat}{3}\label{eq:processes definitions}
    &\alpha^i\in\MVDup&&:\quad&&\alpha_t^i\closedbrac{0,x}\text{ is the mass to have exogenously entered the $i$-th queue by time $t$ }\\
    & && &&\text{with deadlines in $\closedbrac{0,x}$},\nonumber\\
    &\mu^i\in\Dup&&:&&\mu^i\brac{t}\text{ is the total mass the $i$-th server can serve by time $t$, if it is never idle},\nonumber\\
    &\xi^i\in\MVD&&:&&\xi_t^i\closedbrac{0,x}\text{ is the mass in the $i$-th queue at time $t$ with deadlines in $\closedbrac{0,x}$},\nonumber\\
    &\beta^{i}\in\MVDup&&:&&\beta_t^{i}\closedbrac{0,x}\text{ is the mass to have left the $i$-th queue by time $t$ with deadlines in $\closedbrac{0,x}$}.\nonumber
\end{alignat}
The initial condition of the buffer content is already contained in $\alpha$,
namely it is given by $\al_0$. 
Let the idleness (or lost effort)
process be defined by $\iota^i=\mu^i-\beta^i\left[0,\infty\right)$.
The following balance equation holds:
\begin{align}
    \xi_t^i\closedbrac{0,x}&=\alpha_t^i\closedbrac{0,x}+\sum_{j=1}^KP_{ji}\beta_t^j\closedbrac{0,x}-\beta_t^i\closedbrac{0,x}.
\end{align}
The work conservation property and the priority structure of EDF
are expressed through 
\begin{align}
    \int&\xi_t^i\closedbrac{0,x}d\iota^i\brac{t}=0\nonumber,\quad 1\leq i\leq K,\\
    \int&\xi_t^i\closedbrac{0,x}d\beta_t^i\brac{x,\infty}=0,\quad1\leq i\leq K.
\end{align}
One can recognize
similarities to the MVSP, although the objects in our equations are vector-valued
parallelrs
of those from the MVSP, and moreover, the equations are coupled.
This motivates us to define a multi-dimensional version of the MVSP.
\begin{definition}[VMVSP]
Let $P\in\R^{K\times K}$ be a substochastic matrix and let $R=I-P^T$. Let
$\brac{\alpha,\mu}\in\MVDupK\times\DupK$. Then $\brac{\xi,\beta,\iota}\in\MVDK\times\MVDupK\times\DupK$ is said to solve the VMVSP associated with the matrix $P$ for the data $\brac{\alpha,\mu}$ if, for each $x\in\mathbb{R}_+$,
\begin{enumerate}
    \item $\xi\closedbrac{0,x}=\alpha\closedbrac{0,x}-R\beta\closedbrac{0,x}$,
    \item $\int_{\left[0,\infty\right)}\xi_s^i\closedbrac{0,x}d\beta_s^i\brac{x,\infty}=0$ for $i\in\{1,...,K\}$,
    \item $\int_{\left[0,\infty\right)}\xi_s^i\closedbrac{0,x}d\iota^i\brac{s}=0$  for $i\in\{1,...,K\}$,
    \item $\beta\left[0,\infty\right)+\iota=\mu$.
\end{enumerate}
\end{definition}

Note that the motivating fluid model indeed requires that $\xi_t$
be a nonnegative vector measure for each $t$, and that
$t\mapsto\beta_t\brac{x,\infty}+\iota\brac{t}$ be non-decreasing and non-negative. 
These properties are assured by the above notion on VMVSM.

A square matrix is called an $M$-matrix if it has positive diagonal elements, non-positive off-diagonal elements, and a non-negative inverse.
By \cite[Lemma 7.1]{chen2001}, for a non-negative matrix $G$ whose spectral radius is strictly less than 1, also called convergent, $I-G$ is an $M$-matrix. 

\begin{theorem}\label{thm:MD-MVSM}
Let $P\in\R^{K\times K}$ be a convergent substochastic matrix, and $\brac{\alpha,\mu}\in\MVDupK\times\DupK$ be such that $\mu\brac{0}=0$. Then the VMVSP associated with the matrix $P$ and the data $\brac{\alpha,\mu}$ has a unique solution.
\end{theorem}

\begin{proof}
The uniqueness argument is based on SM theory for oblique reflection
in the orthant.
As for existence, our strategy is to construct a candidate
and show that it is a solution,
relying in a crucial way on a monotonicity result due to Ramasubramanian.

We first present the existence argument.
To construct a candidate, first rewrite the first condition of the VMVSP as
\begin{align}\label{eq:soft fluid dynamics vector form}
    \xi_t\closedbrac{0,x}=\alpha_t\closedbrac{0,x}-R\mu\brac{t}+R\brac{\beta_t\brac{x,\infty}+\iota\brac{t}}.
\end{align}
The {\it oblique reflection mapping theorem} (ORMT),
\cite{harrison1981},
\cite[Theorem 7.2]{chen2001}
states that for an M-matrix $R$, for every $u\in\mathbb{D}_0^K$ there exists a unique pair $\brac{z,y}\in\mathbb{D}_+^K\times\DupK$ satisfying
\begin{align}
    \label{eq:ORMT1}&z=u+Ry\\
    \label{eq:ORMT2}&\int_{[0,\iy)} z^i\brac{s}dy^i\brac{s}=0,\quad 1\leq i\leq K.
\end{align}
The solution map, denoted $\Gamma:\mathbb{D}_0^K\to\DK\times\DupK$,
is thus uniquely defined by the relation: $\brac{z,y}=\Gamma\brac{u}=\brac{\Gamma_1\brac{u},\Gamma_2\brac{u}}$ iff \eqref{eq:ORMT1}--\eqref{eq:ORMT2} hold. 
It is well known that the maps $\Gamma_1$ and $\Gamma_2$ are Lipschitz from $\mathbb{D}_0^K$ to $\mathbb{D}_+^{\uparrow K}$ and $\mathbb{D}_+^K$ in the sense that for any $T>0$
\begin{align}\label{eq:Lipschitz of Gamma}
\norm{\Gamma_i\brac{u_1}-\Gamma_i\brac{u_2}}_T\leq L\norm{u_1-u_2}_T,\quad i=1,2,
\end{align}
where the constant $L$ depends only on $R$.

The matrix $P^T$ is non-negative and convergent by assumption, hence $R$ is an $M$-matrix. 
Moreover, $\alpha_0\closedbrac{0,x}-R\mu\brac{0}=\alpha_0\closedbrac{0,x}\geq 0$. 
Also ,if $\brac{\xi,\beta,\iota}$ is a solution, then $\brac{\xi\closedbrac{0,x},\beta\brac{x,\infty}+\iota}\in\DK\times\DupK$.
Finally, Equation \eqref{eq:soft fluid dynamics vector form} has the form of \eqref{eq:ORMT1}, and conditions 2 and 3 of the VMVSP correspond to \eqref{eq:ORMT2}.
Therefore, the conditions of the ORMT are satisfied, so we can construct a candidate as follows. Define
\begin{align}\label{eq:xi and beta candidate}
    &\mathring{\xi}\brac{x}=\Gamma_1\brac{\alpha\closedbrac{0,x}-R\mu}\in\DK, &\mathring{\beta}\brac{x}+\iota=\Gamma_2\brac{\alpha\closedbrac{0,x}-R\mu}\in\DupK.
\end{align}
The path $\iota$ can be recovered by taking $x$ to infinity, yielding $\iota=\Gamma_2\brac{\alpha\left[0,\infty\right)-R\mu}$. 
One can then find $\mathring{\beta}_t\brac{x}$ by $\mathring{\beta}\brac{x}=\Gamma_2\brac{\alpha\closedbrac{0,x}-R\mu}-\iota=\Gamma_2\brac{\alpha\closedbrac{0,x}-R\brac{\mu-\iota}}$.

It is now argued that these functions define vector measure valued paths
via the relations
$(\mathring{\xi}_t\brac{x},\mathring{\beta}_t\brac{x})=
(\xi_t\closedbrac{0,x},\beta_t\brac{x,\infty})$.
We must show that $\mathring{\xi}_t\brac{x}$ is right-continuous
and non-decreasing in $x$, $\mathring{\beta}\brac{x}-\beta\brac{y}$ is in $\DupK$ for $x<y$
and that $\mathring{\beta}\brac{x}$ right-continuous in $x$. 
The right-continuity follows by the Lipschitz property \eqref{eq:Lipschitz of Gamma} and Lemma 2.4 in \cite{atar2018}.
Next, Theorem 4.1 of \cite{Ramasubramanian2000} states that
the map $\Gamma$ is monotone in the following sense. 
Let $u_1,u_2\in\mathbb{D}_0^K$ such that $u_2-u_1\in\DupK$, and let $\brac{z_i,y_i}=\Gamma\brac{u_i},i=1,2$. 
Then $z_2-z_1\in\mathbb{D}_+^K$ and $y_1-y_2\in\DupK$. 
Using this theorem in the setting of the VMVSP,
let $x<y$ and denote $u_1=\alpha\closedbrac{0,x}-R\mu$,
and $u_2=\alpha\closedbrac{0,y}-R\mu$. Then
$u_2-u_1=\alpha\left(x,y\right]\in\DupK$.
Therefore $\mathring{\xi}\brac{y}-\mathring{\xi}\brac{x}\in\mathbb{D}_+^K$ and $\mathring{\beta}\brac{x}-\mathring{\beta}\brac{y}\in\DupK$. 
Thus
\begin{align}\label{eq:xi and beta}
    &\xi\closedbrac{0,x}=\mathring{\xi}\brac{x}=\Gamma_1\brac{\alpha\closedbrac{0,x}-R\mu},
    &\beta\brac{x,\infty}=\mathring{\beta}\brac{x}=\Gamma_2\brac{\alpha\closedbrac{0,x}-R\mu}-\iota,
\end{align}
define two vector measure valued paths in $\MVDK$ and $\MVDupK$ respectively.
This demonstrates existence.

As for uniqueness, it follows from the ORMT that any solution
to the VMVSP must satisfy
\[
(\xi\closedbrac{0,x},\beta(x,\iy)+\iota)=\Gamma\brac{\alpha\closedbrac{0,x}-R\mu},
\]
a relation that defines uniquely $\xi,\beta$ and $\iota$.
\end{proof}

Theorem \ref{thm:MD-MVSM} defines a map from $\MVDupK\times\DupK\to\MVDK\times\MVDupK\times\DupK$,
namely the solution map of the VMVSM. We denote it by $\Theta$. The
dependence on $P$ is not indicated explicitly.
This notion is an extension of the MVSM from \cite{atar2018}.

\subsubsection{Queueing model and scaling} \label{sec:soft prelimit}

The queueing model is defined analogously to the fluid model.
It is indexed by $N\in\mathbb{N}$, and defined on
a probability space $(\Om,\calF,\PP)$.
It consists of $K\ge1$ service stations, each containing a buffer with infinite room and a server. 
The servers prioritize according to the EDF policy, and according to arrival times in case of a tie.
Throughout this paper, for any parameter, random variable or process
associated with the $N$-th system, say $a^N$,
we use the notation $\bar{a}^N=N^{-1}a^N$ for normalization.
A substochastic matrix $P\in\R^{K\times K}$, referred to as the routing matrix, is given,
where the entry $P_{ij}$ represents the probability that a job that completes
service at the $i$-th server is routed to the $j$-th queue.
Denote $R=I-P^T$.
For each $N$, processes denoted by $\alpha^N=(\al^{i,N})$, $\mu^N=(\mu^{i,N})$,
$\xi^N=(\xi^{i,N})$, and $\beta^N=(\beta^{i,N})$ are associated with the $N$-th system,
which represent discrete versions of their fluid model counterparts.
Specifically, for a Borel set $B\subset\R_+$, $\al^{i,N}_t(B)$
denotes the number of external arrivals into queue $i$ up to time $t$
with deadline in $B$, $\xi^{i,N}_t(B)$ denotes the number of jobs in queue $i$ at time $t$ with deadline in $B$, and $\beta_t^{i,N}\brac{B}$
denotes the number of jobs with deadline in $B$ transferred by time $t$ from the $i$-th queue
to the corresponding server.
Indeed, in the queueing model, the job being served is not counted in the queue,
and therefore there is a distinction between the number of jobs transferred from queue $i$
and the number of jobs transferred from server $i$. Thus in addition to $\beta^{i,N}$ we introduce
a process $\gamma^N$. Namely, $\gamma^{ij,N}_t(B)$ denotes the number of jobs with deadline in $B$
that were transferred from server $i$ to queue $j$ by time $t$, where $j=0$ corresponds to jobs
leaving the system.
For each $i$, $\gamma^{i,N}_t(B):=\sum_{j=0}^K\gamma^{ij,N}_t(B)$ gives the total number of jobs with deadline in $B$ that departed server $i$ by time $t$.
The process $\mu^{i,N}\brac{t}$ represents the cumulative service capacity of server $i$ by time $t$.
Let the busyness at time $t$, $B_t^{i,N}\brac{B}$, be defined as the indicator
of the event that at this time, a job with deadline in $B$ occupies the $i$-th server, and let $B_{0-}^{i,N}\brac{B}$ be its initial condition.
Denote the total number of departures from server $i$ by
$D^{i,N}\brac{t}:=\gamma^{i,N}_t\left[0,\infty\right)$. 
Denote by $\xi_{0-}^{i,N}\brac{B}$ the number of jobs present in the $i$-th queue just prior to time $t=0$ with deadlines in $B$.

Next, define the cumulative effort and cumulative lost effort processes, respectively, as
\begin{align}
    &T^{i,N}\brac{t}=\int_{[0,t]}B_s^{i,N}\lt[0,\infty\rt)d\mu^{i,N}\brac{s},\\
    &\iota^{i,N}\brac{t}=\mu^{i,N}\brac{t}-T^{i,N}\brac{t}.
\end{align}
To model the stochasticity of the service times
a counting process $S^i\brac{t}$ associated with each server $i$ is introduced, assumed to be
a renewal process for which the interarrival times have unit mean,
with the convention $S^i\brac{0}=1$. 
Then the departure process is given by
\begin{align}\label{eq:soft prelimit departures}
    D^{i,N}\brac{t}=\gamma^{i,N}_t\left[0,\infty\right)=S^i\brac{T^{i,N}\brac{t}}-1.
\end{align}

The various relations between these processes are described in what follows.
The relation between $\gamma^N$, $\beta^N$ and $B^N$ is given by
\begin{align}\label{eq:soft fluid beta gamma}
    B_{0-}^{i,N}\closedbrac{0,x}+\beta_t^{i,N}\closedbrac{0,x}=\gamma_t^{i,N}\closedbrac{0,x}+B_t^{i,N}\closedbrac{0,x}.
\end{align}
The balance equation for the queue content is
\begin{align}\label{eq:soft prelimit dynamics}
    \xi_t^{i,N}\closedbrac{0,x}&=\alpha_t^{i,N}\closedbrac{0,x}+\sum_{j=1}^K\gamma_t^{ji,N}\closedbrac{0,x}-\beta_t^{i,N}\closedbrac{0,x}.
\end{align}
In addition, the work conservation condition and EDF policy are expressed through
\begin{alignat}{2}\label{eq:prelimit soft work conservation conditions}
    \int_{[0,\iy)}&\xi_t^{i,N}\closedbrac{0,x}d\iota^{i,N}\brac{t}=0,\quad&&1\leq i\leq K,\\
    \label{eq:prelimit soft EDF conditions}\int_{[0,\iy)}&\xi_t^{i,N}\closedbrac{0,x}d\beta_t^{i,N}\brac{x,\infty}=0,\quad&&1\leq i\leq K.
\end{alignat}

For the routing process of server $i$, consider i.i.d.\ $\{0,1,\ldots,K\}$-valued
random variables $\pi^{i,N}\brac{n}$ with distribution given by
\begin{align*}
    \prob{\pi^{i,N}\brac{n}=j}=
    \begin{cases}
        1-\sum_{k=1}^KP_{ik} & \text{if }j=0, \\
        P_{ij} & \text{if } j\in\{1,...,K\}.
    \end{cases}
\end{align*}
The stochastic primitives
$\{\pi^{i,N}\brac{n}\}$, $\{S\}$ and $\{\alpha^N\}$ are assumed to be mutually independent.
Moreover, $\pi^{i,N}\brac{n}$ are assumed to be mutually independent across $i$, and so are
$S^{i}$ and $\al^{i,N}$.

Define for $j\in\{1,...,K\}:$  $\theta^{ij,N}\brac{n}=\indicator{\pi^{i,N}\brac{n}=j}$.
The $n$-th job to depart server $i$ is routed to server $\pi^{i,N}\brac{n}$,
or, if this random variable is zero, leaves the system.
Thus, the $n$-th job is routed to server $j$ if and only if $\theta^{ij,N}\brac{n}=1$.
Let the jump times of $D^{i,N}$ be denoted by
\begin{align}
    \tau_n^{i,N}=\inf\{t\geq 0:D^{i,N}\brac{t}\geq n\},
\end{align}
and let
\begin{align}
    \hat{\theta}^{ij,N}\brac{t}=\sum_{n=1}^\infty\mathbbm{1}_{\lt[\tau_{n}^{i,N},\tau_{n+1}^{i,N}\rt)}\brac{t}\theta^{ij,N}\brac{n}.
\end{align}
Then $\gamma^{ij,N}$ can be obtained from $\gamma^{i,N}$ by 
\begin{align}\label{104}
    \gamma_t^{ij,N}\closedbrac{0,x}=\int_{\lt[0,t\rt]}\hat{\theta}^{ij,N}\brac{s}d\gamma_s^{i,N}\closedbrac{0,x}.
\end{align}
This completes the definition of the model.

Our main result concerning soft EDF networks is the following.
\begin{theorem}\label{thm: soft convergence}
Assume that the routing matrix $P$ is a convergent substochastic matrix,
and let
$\brac{\alpha,\mu}\in\MVCupK\times\mathbb{C}^{\uparrow K}$ be such that $\mu\brac{0}=0$.
Assume, moreover, that $\al_t\in\calM_{\sim}$ for all $t$,
and that there exists a constant $C$
such that $\expect{\sum_{i=1}^K\alpha_t^{i,N}\lt[0,\infty\rt)}\leq CNt$.
Finally, assume $(\bar\al^N,\bar\mu^N)\To(\al,\mu)$.
Then $\brac{\bar{\xi}^N,\bar{\beta}^N,\bar{\iota}^N}\Rightarrow\brac{\xi,\beta,\iota}$,
where $\brac{\xi,\beta,\iota}$ is the unique solution of the VMVSP with primitives $\brac{\alpha,\mu}$.
\end{theorem}
The proof is given in the next section.

\subsection{Proof}\label{sec: analysis & proofs}

Recall the ORMT. To use this theorem, we wish to bring equations \eqref{eq:soft prelimit dynamics}, \eqref{eq:prelimit soft work conservation conditions} and \eqref{eq:prelimit soft EDF conditions} to a form compatible with the conditions \eqref{eq:ORMT1} and \eqref{eq:ORMT2}. To this end,
first define the error processes
\begin{align}\label{eq:errors}
    &e^{i,N}\brac{t}=\beta_t^{i,N}\left[0,\infty\right)+\iota^{i,N}\brac{t}-\mu^{i,N}\brac{t},\quad 1\leq i\leq K,\nonumber\\
    &E^{ij,N}\brac{t,x}=\gamma_t^{ij,N}\closedbrac{0,x}-P_{ij}\gamma_t^{i,N}\closedbrac{0,x},\quad 1\leq i,j\leq K, x\in\mathbb{R}_+.
\end{align}
Using \eqref{eq:soft prelimit dynamics} and normalizing,
\begin{align*}
    \bar{\xi}_t^{i,N}\closedbrac{0,x}&=\bar{\alpha}_t^{ij,N}\closedbrac{0,x}+\sum_{j=1}^K\bar{E}^{ji,N}\brac{t,x}+\sum_{j=1}^KP_{ji}\brac{\bar{\mu}^{j,N}\brac{t}+\bar{e}^{j,N}\brac{t}}-\bar{\mu}^{i,N}\brac{t}-\bar{e}^{i,N}\brac{t}\\
    &\quad-\sum_{j=1}^KP_{ji}\brac{\bar{B}_t^{j,N}\closedbrac{0,x}-\bar{B}_{0-}^{j,N}\closedbrac{0,x}}+\bar{\beta}_t^{i,N}\brac{x,\infty}+\bar{\iota}^{i,N}\brac{t}-\sum_{j=1}^KP_{ji}\brac{\bar{\beta}_t^{j,N}\brac{x,\infty}+\bar{\iota}^{j,N}\brac{t}}.
\end{align*}
In vector notation, with $\textbf{1}=(1,\ldots,1)^T\in\R^K$,
\begin{align}\label{eq:soft PL dynamics vector}
    \bar{\xi}_t^N\closedbrac{0,x}=&\bar{\alpha}_t^N\closedbrac{0,x}+\bar{E}^{N,T}\brac{t,x}\textbf{1}-R\brac{\bar{\mu}^N\brac{t}+\bar{e}^N\brac{t}}\\
    &-P^T\brac{\bar{B}_t^N\closedbrac{0,x}-\bar{B}_{0-}^N\closedbrac{0,x}}+R\brac{\bar{\beta}_t^{N}\brac{x,\infty}+\bar{\iota}^N\brac{t}}.\nonumber
\end{align}
The matrix $R$ is an M-matrix as already argued, $\bar{\xi}^N$ is non-negative, and $\bar{\beta}_t^N\brac{x,\infty}+\bar{\iota}^N$ is non-decreasing and non-negative. 
We can now invoke the ORMT, yielding for every $x$,
\begin{align}\label{eq:xi^N and beta^N}
    &\bar{\xi}^N\closedbrac{0,x}=\Gamma_1\brac{\bar{\alpha}^N\closedbrac{0,x}+\bar{E}^{N,T}\brac{\cdot,x}\textbf{1}-R\brac{\bar{\mu}^N+\bar{e}^N}-P^T\bar{B}^N\closedbrac{0,x}+P^T\bar{B}_{0-}^N\closedbrac{0,x}},\nonumber\\
    &\bar{\beta}^N\brac{x,\infty}+\bar{\iota}^N=\Gamma_2\brac{\bar{\alpha}^N\closedbrac{0,x}+\bar{E}^{N,T}\brac{\cdot,x}\textbf{1}-R\brac{\bar{\mu}^N+\bar{e}^N}-P^T\bar{B}^N\closedbrac{0,x}+P^T\bar{B}_{0-}^N\closedbrac{0,x}}.
\end{align}

\begin{lemma}\label{lem:martingale}
Let $E^{ij,N}\brac{t,x}$ and $D^{i,N}\brac{t}$ be as in \eqref{eq:errors} and \eqref{eq:soft prelimit departures}. Then, for all $i,j,N,t,x$,
    \begin{align*}
        \expect{E^{ij,N}\brac{t,x}^2}\leq 41\expect{D^{i,N}\brac{t}}.
    \end{align*}
\end{lemma}
\begin{proof}
We have
\begin{align*}
    E^{ij,N}\brac{t,x}&=\gamma_t^{ij,N}\closedbrac{0,x}-P_{ij}\gamma_t^{i,N}\closedbrac{0,x}\\
    &=\int_{\left[0,t\right]}\brac{\hat{\theta}^{ij,N}\brac{s}-P_{ij}}d\gamma_s^{i,N}\closedbrac{0,x}\\
    &=\sum_{n=1}^{D^{i,N}\brac{t}}\brac{\theta^{ij,N}\brac{n}-P_{ij}}\brac{\gamma_{\tau_{n}^{i,N}}^{i,N}\closedbrac{0,x}-\gamma_{\tau_{n-1}^{i,N}}^{i,N}\closedbrac{0,x}}.
\end{align*}
To simplify notation in this proof, we omit in what follows
the dependence on $i,j$ and $N$. Denote
\begin{align*}
    M_{k}\brac{x}&=\sum_{n=1}^{k}\brac{\theta\brac{n}-P}\brac{\gamma_{\tau_{n}}\closedbrac{0,x}-\gamma_{\tau_{n-1}}\closedbrac{0,x}}.
\end{align*}
We show that $M_k(x)$ is a martingale on the filtration
\begin{align*}
    \calig{F}_k=\sigma\brac{\theta\brac{n}: 1\le n\le k,\ 
\gamma_{\tau_n}\closedbrac{0,y}: 1\leq n\leq k+1,
y\geq 0,\ 
\tau_n: 1\le n\le k+1}.
\end{align*}
We first argue that $\theta\brac{k+1}$ is independent of $\calig{F}_{k}$. 
By assumption, $\theta\brac{k+1}$ is independent of
\[
\calig{G}_k:=\brac{\theta\brac{n},S\brac{t},\alpha_t\closedbrac{0,y},\quad t,y\geq 0,1\leq n\leq k}.
\]
By construction, $\calig{G}_k\supset\calig{F}_k$.
Therefore, $\theta\brac{k+1}$ is independent of $\calig{F}_{k}$. 
The process $M_k\brac{x}$ is adapted to $\{\calig{F}_k\}_{k\geq 1}$,
it satisfies $\expect{\abs{M_k\brac{x}}}\leq k<\infty$, and
\begin{align*}
    &\E\lt[\sum_{n=1}^{k+1}\brac{\theta\brac{n}-P}\brac{\gamma_{\tau_{n}}\closedbrac{0,x}-\gamma_{\tau_{n-1}}\closedbrac{0,x}}\middle|\calig{F}_k\rt]\\
    &=M_k\brac{x}+\brac{\gamma_{\tau_{k+1}}\closedbrac{0,x}-\gamma_{\tau_{k}}\closedbrac{0,x}}\E\lt[\theta\brac{k+1}-P\middle|\calig{F}_k\rt]\\
    &=M_k\brac{x}.
\end{align*}
Consequently it is a martingale.

Next, $D\brac{t}$ is a stopping time with respect to $\{\calig{F}_k\}_{k\geq 1}$, because $\{D\brac{t}\leq k\}=\{\tau_{k+1}> t\}$ and $\tau_{k+1}$ is $\calig{F}_k$-measurable.
Therefore, $M_{k\land D\brac{t}}$ is also a martingale, and by
Fatou's lemma and Burkholder's inequality (see \cite{hall1980}, Theorem 2.10)
\begin{align*}
    \expect{E^2\brac{t,x}}&=\expect{M_{D\brac{t}}^2\brac{x}}\\
    &\leq\liminf_{k\to\infty}\expect{M_{k\land D\brac{t}}^2\brac{x}}\\
    &\leq 41\liminf_{k\to\infty}\expect{\sum_{n=1}^{k\land D\brac{t}}\brac{\theta\brac{n}-P}^2\brac{\gamma_{\tau_n}\closedbrac{0,x}-\gamma_{\tau_{n-1}}\closedbrac{0,x}}^2}\\
    &\le 41\expect{D\brac{t}}.
\end{align*}
\end{proof}

\begin{proof}[Proof of Theorem \ref{thm: soft convergence}]
We write $c$ for a generic positive constant whose value may change from line to line.

We first show that $\bar{\xi}^N\to\xi$ in probability.
If $\al_1,\al_2\in\calM$ then $d_\calL(\al_1,\al_2)\le\sup_x|\al_1[0,x]-\al_2[0,x]|$.
Consequently it suffices to show for every $T$
\begin{align}\label{eq:supx xi converges}
    \sup_x\norm{\bar{\xi}^{N}\closedbrac{0,x}-\xi\closedbrac{0,x}}_T\to 0 \text{ in probability}.
\end{align}
By \eqref{eq:Lipschitz of Gamma}, \eqref{eq:xi and beta} and \eqref{eq:xi^N and beta^N},
\begin{align*}
    &\sup_x\norm{\bar{\xi}^{N}\closedbrac{0,x}-\xi\closedbrac{0,x}}_T\\
    &\leq L\sup_x\norm{\bar{\alpha}^N\closedbrac{0,x}-\alpha\closedbrac{0,x}}_T+L\sup_x\norm{\bar{E}^{N,T}\brac{\cdot,x}\textbf{1}}_T+L\sup_x\norm{P^T\bar{B}^N\closedbrac{0,x}+P^T\bar{B}_{0-}^N\closedbrac{0,x}}_T\\
    &\hspace{20pt}+L\norm{R\brac{\mu-\bar{\mu}^N}}_T+L\norm{R\bar{e}^N}_T.
\end{align*}
We now show that each of the terms converges in probability to 0.\\
    \textbullet\
It follows directly from the definition of the L\'evy metric that for
$a_1, a_2\in{\cal M}$,
\begin{align}\label{eq:d dominations}
\sup_{x\in\mathbb{R_+}}\abs{a_1\closedbrac{0,x}-a_2\closedbrac{0,x}}\leq d_{\calig{L}}\brac{a_1,a_2}+w\brac{a_2\closedbrac{0,\cdot},2d_{\cal L}\brac{a_1,a_2}}.
\end{align}
The fact that $\al^i\in\mathbb{D}_\calM^\uparrow$ implies that
for any $t\leq T$, $w\brac{\alpha_t\closedbrac{0,\cdot},\delta}\leq w\brac{\alpha_T\closedbrac{0,\cdot},\delta}$.
Since $\al_T\in\calM$ does not have atoms by assumption,
we have $w\brac{\alpha_T\closedbrac{0,\cdot},\epsilon}\to0$ as $\epsilon\to0$. 
Denote
\begin{align*}
    d_{\calig{L},T}\brac{\zeta_1,\zeta_2}=\sup_{t\in\closedbrac{0,T}}\max_{1\leq i\leq K}d_{\cal L}\brac{\zeta_{1,t}^i,\zeta_{2,t}^i},
\end{align*}
for $\zeta_1,\zeta_2\in\MVDK$.
By the convergence $\al^N\To\al$ and the continuity of $t\mapsto\al_t$,
it follows that $d_{\calig{L},T}\brac{\alpha^N,\alpha}\rightarrow{0}$ in probability.
Hence by \eqref{eq:d dominations},
    \begin{align*}
        \sup_x\norm{\bar{\alpha}^N\closedbrac{0,x}-\alpha\closedbrac{0,x}}_T\to 0\text{ in probability}.
    \end{align*}
\\
    \textbullet\
    For the second term, it suffices to show
    \begin{align*}
\sup_x\sup_{t\in\closedbrac{0,T}}\max_{1\leq i,j\leq K}\abs{\bar{E}^{ji,N}\brac{t,x}}\to 0\text{ in probability}.
    \end{align*}
To this end note that Lemma \ref{lem:martingale} implies 
    \begin{align*}
        \expect{\abs{\bar{E}^{ij,N}\brac{t,x}}^2}\leq 41N^{-2}\expect{D^{i,N}\brac{t}}\leq 41N^{-2}\expect{\sum_{i=1}^K\alpha^{i,N}_T\lt[0,\infty\rt)}\leq 41N^{-1}CT\to 0.    
    \end{align*}
    \textbullet\  The third term on the RHS
     is bounded by $2LK/N\to 0$.\\
    \textbullet\
Next, $\norm{R\brac{\mu-\bar{\mu}^N}}_T\leq c\norm{\mu-\bar{\mu}^N}_T\to 0$
in probability by the assumed convergence of $\bar\mu^N$ to $\mu$
and continuity of $\mu$.
\\
    \textbullet\
    Finally, by the definition of $\bar e^N$,
    \begin{align*}
        \bar{e}^{i,N}\brac{t}&=\bar{\beta}_t^{i,N}\left[0,\infty\right)+\bar{\iota}^{i,N}\brac{t}-\bar{\mu}^{i,N}\brac{t}\\
        &=-\frac{1}{N}+\frac{B^{i,N}\lt[0,\infty\rt)}{N}+\frac{S^i\brac{N\bar T^{i,N}\brac{t}}-N\bar{T}^{i,N}\brac{t}}{N}.
    \end{align*}
Now, by the law of large numbers for renewal processes, for any constant $c$,
$\sup_{t\in\closedbrac{0,c}}\abs{\frac{S^i\brac{Nt}-Nt}{N}}\to 0$ in probability.
Moreover, for fixed $T$, $\bar T^{i,N}(T)$ are tight. It follows that
\[
\sup_{t\in[0,T]}\Big|\frac{S^i\brac{N\bar T^{i,N}\brac{t}}-N\bar{T}^{i,N}\brac{t}}{N}\Big|\to0
\]
in probability. Consequently, $\|\bar e^N\|_T\to0$ in probability.
This completes the proof of \eqref{eq:supx xi converges}. We conclude that $\bar\xi^N\to\xi$.

The convergence of $\bar{\iota}^N$ and $\bar{\beta}^N$ to $\iota$ and $\beta$ is shown
similarly, by virtue of the Lipschitz property of $\Gamma_2$.
We omit the details. This shows
$(\bar\xi^N,\bar\beta^n,\bar\iota^N)\to(\xi,\beta,\iota)$ in probability,
and completes the proof.
\end{proof}

\section{Hard EDF networks}\label{sec3}

In \S\ref{sec31} the fluid model equations and the queueing model are
presented, as well as the main results, asserting uniqueness
and convergence. In \S\ref{sec32} auxiliary lemmas are established.
The proofs of the main results are presented in \S\ref{sec33}.

\subsection{Model and main results}\label{sec31}
\subsubsection{Fluid model} \label{sec:Hard fluid}

The hard version of the policy differs from the soft version in that, when the deadline expires,
fluid in the queue reneges. Therefore, we distinguish between the reneging process and the departure process:
\begin{alignat*}{2}
    &\beta^{is}\in\MVDup:\quad&&\beta_t^{is}\closedbrac{0,x}\text{ is the mass with deadlines in $\closedbrac{0,x}$ that has left the $i$-th queue}\\
    & && \text{by time $t$ due to service},\\
    &\beta^{ir}\in\MVDup:&&\beta_t^{ir}\closedbrac{0,x}\text{ is the mass with deadlines in $\closedbrac{0,x}$ that has left the $i$-th queue}\\
    & && \text{by time $t$ due to reneging}.
\end{alignat*}
The deadline is not fixed but is postponed by a constant
$\epsilon>0$ whenever mass migrates to a new service station.
We introduce $\gamma=\brac{\gamma^i}$ defined by
\begin{equation}\label{100}
\gamma_t^{i}\closedbrac{0,x}=\beta_t^{i,s}\closedbrac{0,x-\epsilon},
\end{equation}
(recalling the convention $[a,b]=\emptyset$ if $a>b$),
which expresses the mass that has left server $i$ by time $t$ with
its updated deadline in $\closedbrac{0,x}$.
The processes $\alpha$, $\mu$, $\xi$, $\iota$ are defined
as before.

We now describe the equations governing the fluid model. The fluid that is served and the fluid that reneges from the system sum up to the total fluid that leaves the queue, and the idle effort is the difference between the potential effort and the processed mass, as described by
\begin{align}
\label{107}
    &\beta_t^i\closedbrac{0,x}=\beta_t^{i,s}\closedbrac{0,x}+\beta_t^{i,r}\closedbrac{0,x}, \\
\label{108}
    &\iota^i\brac{t}=\mu^i\brac{t}-\beta_t^{i,s}\left[0,\infty\right).
\end{align}
Reneging does not occur prior to deadline,
as expressed by
\begin{equation}\label{eq:renege after deadline}
\beta_t^r\brac{t,\infty}=0.
\end{equation}
The content of the $i$-th queue is given by
\begin{align}
    \xi_t^i\closedbrac{0,x}&=\alpha_t^i\closedbrac{0,x}+\sum_{j=1}^KP_{ji}\gamma_t^{j,s}\closedbrac{0,x}-\beta_t^i\closedbrac{0,x}\label{eq:Hard fluid dynamics}.
\end{align}
The work conservation and EDF conditions are expressed through
\begin{alignat}{2}
    &\int\xi_t^i\closedbrac{0,x}d\iota^i\brac{t}=0,&&\quad 1\leq i\leq K,\label{eq:hard fluid work cons}\\
    &\int\xi_t^i\closedbrac{0,x}d\beta_t^{i}\brac{x,\infty}=0,&&\quad1\leq i\leq K\label{eq:hard fluid EDF cond}.
\end{alignat}
To complete the description we define the processes
\begin{align}
    &\rho^{i}\brac{t}=\beta_t^{i,r}\closedbrac{0,t}=\beta_t^{i,r}\left[0,\infty\right),\label{eq:rho def}\\
    &\sigma^i\brac{t}=\inf\text{supp}\closedbrac{\xi_t^i}.
\label{101}
\end{align}
Note that $\beta_t^{ir}\closedbrac{0,x}=\rho^i\brac{t\land x}$.
The quantity $\rho^i(t)$ expresses the cumulative reneging from buffer $i$.
The significance of $\sigma^i(t)$, the left edge of the support of $\xi^i$,
is that reneging from buffer $i$ occurs only at times $t$ when $\sigma^i(t)=t$.
Thus
\begin{alignat}{2}
    &\xi_t^i\left[0,t\right)=0,&&\quad1\leq i\leq K,\label{eq:hard fluid hardness}\\
    &\int\indicator{\sigma^i\brac{t}>t}d\rho^i\brac{t}=0,&&\quad 1\leq i\leq K. 
    \label{103}
\end{alignat}
We refer to \eqref{100}--\eqref{103} as the {\it fluid model equations}.
A solution to these equations for data $(\al,\mu)$ is a tuple
$(\xi,\beta,\beta^s,\beta^r,\rho,\iota)$ for which these equations hold.
Note that it is possible to recover $\beta$ and $\rho$ from $\beta^s$ and $\beta^r$,
and vice versa.
Therefore, we sometimes refer to a tuple $(\xi,\beta,\rho,\iota)$
or $(\xi,\beta^s,\beta^r,\iota)$ as a solution.

\begin{assumption}\label{ass:data assumptions}
\begin{enumerate}
    \item $\al$ takes the form
$\alpha_t(B)=\xi_{0-}(B)+\hat{\alpha}_t(B)$,
where $\xi_{0-}\in{\cal M}_{\sim}^K$, and $\hat{\alpha}\in\mathbb{C}_{{\cal M}_\sim}^{\uparrow K}$ satisfies for every $1\leq i\leq K$
    \begin{align*}
        \hat{\alpha}_t^i(B)=\int_0^ta^i_s(B)ds,\quad t\geq 0,B\in\calB(\R_+),
    \end{align*}
where for each $t$ and $i$, $a^i_t$ is a finite measure on $[t,\iy)$, and
$t\mapsto a^i_t(B)$ is measurable.
Moreover, $\lim_{\del\downarrow0}\sup_{s\in[0,t]}a^i_s[s,s+\del]=0$.
\item $\mu$ takes the form
    \begin{align*}
        \mu\brac{t}=\int_0^tm\brac{s}ds,\quad t\geq 0,
    \end{align*}
for a Borel measurable function $m:\left[0,\infty\right)\to\mathbb{R}_+^K$
satisfying $\min_i\inf_{s\in\closedbrac{0,t}}m^i\brac{s}> 0$ for every $t$.
\end{enumerate}
\end{assumption}
\begin{theorem}\label{thm:uniqueness hard}
Suppose $\brac{\alpha,\mu}$ satisfies Assumption \ref{ass:data assumptions}. Then there exists a unique solution $\brac{\xi,\beta,\beta^s,\beta^r,\rho,\iota}$ in $\mathbb{D}_{{\cal M}}^{K}\times(\mathbb{D}_{{\cal M}}^{K\uparrow})^3\times(\mathbb{D}_{+}^{K\uparrow})^2$ to the fluid model. Moreover, $\beta^s$ satisfies Assumption \ref{ass:data assumptions}.1.
\end{theorem}
We prove this result in \S\ref{sec33}.
\subsubsection{Queueing model, scaling}
There are many ingredients of the model that are the same as in \S\ref{sec2}, which
we will not repeat.
%
The processes $\alpha_t^{i,N}$, $\mu^{i,N}\brac{t}$, $\xi_t^{i,N}$,
$\iota^{i,N}(t)$ have the same meaning as in the soft EDF model.
$\beta_t^{i,s,N}\closedbrac{0,x}$ is the number of jobs that have left the $i$-th queue by time $t$ with deadlines in $\closedbrac{0,x}$ to get served,
$\beta_t^{i,r,N}\closedbrac{0,x}$ is the number of jobs that have reneged from the $i$-th queue by time $t$ with deadlines in $\closedbrac{0,x}$ when their deadlines had expired.
Also, the cumulative reneging count is denoted by
$\rho^{i,N}\brac{t}$ and satisfies $\rho^{i,N}\brac{t}=\beta_t^{i,r,N}\lt[0,\infty\rt)$.

As already mentioned, the deadline of a job is postponed by $\eps$
every time it migrates from one station to another.
It is possible that the deadline of a job at a station expires while
it is in service at that station. In this case, it is assumed that
the job does not renege
(in agreement with the model from \cite{atar2018}).
A more complicated scenario is when the work associated with
this job is sufficiently large that it is not complete even
$\eps$ units of time after this deadline expires. At this time
the deadline at the next station also expires.
We prefer not to deal with this possibility because it complicates
the notation. Because $\eps$ is fixed whereas the service times
are downscaled, an event like that becomes less and less probable
as $N$ increases. Our assumption, that will allow us to avoid this scenario
altogether, is that the service time distributions have bounded supports.
This assures that for $N$ large enough this scenario does not occur.
The description of the model equations is given for sufficiently large $N$.

Next, the processes $S^i(t)$, $D^{i,N}(t)$, $\gamma_t^{ij,N}$, $B_t^{i,N}$
and $T^{i,N}(t)$
as well as the relations between them are as in the soft EDF model.
Note that the assumption regarding service time distribution supports
is really an assumption about the processes $S^i$.


We turn to mathematically describe the relations between the processes.
We have
\begin{align*}
    &T^{i,N}\brac{t}=\int_0^tB_s^{i,N}\lt[0,\infty\rt)d\mu^{i,N}\brac{s},\\
    &\iota^{i,N}\brac{t}=\mu^{i,N}\brac{t}-T^{i,N}\brac{t}.
\end{align*}
\begin{align}
    &\beta_t^{i,N}\closedbrac{0,x}=\beta_t^{is,N}\closedbrac{0,x}+\beta_t^{ir,N}\closedbrac{0,x},\\
    &D^{i,N}\brac{t}=\gamma^{i,N}_t\left[0,\infty\right)=S^i\brac{T^{i,N}\brac{t}}-1,\\
    &\gamma_t^{i,N}\closedbrac{0,x}=\beta_t^{i,s,N}\closedbrac{0,x-\epsilon}-B_t^{i,N}\closedbrac{0,x-\epsilon}+B_{0-}^{i,N}\closedbrac{0,x-\epsilon}.\label{eq:hard prelimit gamma beta}
\end{align}
\begin{align}\label{eq:hard Jackson dynamics}
    \xi_t^{i,N}\closedbrac{0,x}&=\alpha_t^{i,N}\closedbrac{0,x}+\sum_{j=1}^K\gamma_t^{ji,N}\closedbrac{0,x}-\beta_t^{i,N}\closedbrac{0,x}.
\end{align}
The work conservation condition, the EDF policy, the hard EDF condition and the latest reneging condition are expressed through
\begin{alignat}{2} 
    \int&\xi_t^{i,N}\closedbrac{0,x}d\iota^{i,N}\brac{t}=0,&&\quad1\leq i\leq K\label{eq:hard prelimit work cons}\\
    \int&\xi_t^{i,N}\closedbrac{0,x}d\beta_t^{i,N}\brac{x,\infty}=0,&&\quad1\leq i\leq K,\label{eq:hard prelimit EDF cond}\\
    &\xi_t^{i,N}\left[0,t\right]=0,&&\quad 1\leq i\leq K,\label{eq:hard prelimit hardness}\\
    \int&\mathbbm{1}_{\{\sigma^{i,N}\brac{t}>t\}}d\rho^{i,N}\brac{t}=0,&&\quad1\leq i\leq K,
\end{alignat}
where $\sigma^{i,N}(t)=\inf\text{supp}\,\xi_t^{i,N}$.
Finally, the assumptions on $\pi^{i,N}\brac{n}$, the definition of
$\theta^{ij,N}\brac{n}$, $\tau_n^{i,N}$, $\hat{\theta}^{ij,N}\brac{t}$
and the relations between these processes are as in the soft EDF model.

Let us define the error processes as
\begin{align}\label{eq:errors hard}
    &e^{i,N}\brac{t}=\beta_t^{i,s,N}\left[0,\infty\right)+\iota^{i,N}\brac{t}-\mu^{i,N}\brac{t},\quad 1\leq i\leq K,\nonumber\\
    &E^{ij,N}\brac{t,x}=\gamma_t^{ij,N}\closedbrac{0,x}-P_{ij}\gamma_t^{i,N}\closedbrac{0,x},\quad 1\leq i,j\leq K, x\in\mathbb{R}_+.
\end{align}

\begin{theorem}\label{thm:hard convergence}
Assume $(\bar\al^N,\bar\mu^N)\To(\al,\mu)$, where $(\al,\mu)$
satisfies Assumption \ref{ass:data assumptions}.
Then 
\begin{align}\label{112}
    \brac{\bar{\xi}^N,\bar{\beta}^N,\bar\beta^{s,N},\bar\beta^{r,N},\bar{\iota}^N,\bar{\rho}^N,\bar{e}^N,\bar{E}^N}\To\brac{\xi,\beta,\beta^s,\beta^r,\iota,\rho,0,0},
\end{align}
where $\brac{\xi,\beta,\beta^s,\beta^r,\iota,\rho}$ is the unique solution of the fluid model equations
corresponding to $(\al,\mu)$.
\end{theorem}

\subsection{Auxiliary lemmas}\label{sec32}

When setting $K=1$ and $P=0$, the model coincides with the single server model as described in \cite{atar2018}, for which uniqueness of the solution of the fluid model equation and convergence have been established in \cite{atar2018}.
\begin{lemma}
\label{thm:4.10 in atar18}
Let $K=1$ and $P=0$.
Suppose $\brac{\alpha,\mu}$ satisfies Assumption \ref{ass:data assumptions}.
Then there exists a unique solution $(\xi,\beta,\iota,\rho)$ to the fluid model equations,
and the solution lies in
$\MVCnoatoms\times\MVCupnoatoms\times\CCup\times\CCup$.
\label{thm:5.4 in atar18}
Assume, moreover, that $(\bar\al^N,\bar\mu^N)\To(\al,\mu)$.
Then
$\brac{\bar{\xi}^N,\bar{\beta}^N,\bar{\iota}^N,\bar{\rho}^N}\Rightarrow\brac{\xi,\beta,\iota,\rho}$.
\end{lemma}
\proof This is the content of Theorems 4.10 and 5.4 in \cite{atar2018}.
\qed


\begin{lemma}\label{lem:rho nonanticipation}
Suppose $\brac{\alpha,\mu}$
satisfies Assumption \ref{ass:data assumptions} and fix $\tau>0$. Let
$\alpha^\circ\in\mathbb{D}_{\cal M}^\uparrow$ be defined by
$\alpha_t^\circ\brac{A}=\alpha_t\brac{A\cap\closedbrac{0,\tau}}$
for all $A\in\calB(\R_+)$ and $t\in\R_+$.
Let $\brac{\xi^\circ,\beta^\circ,\iota^\circ,\rho^\circ}$ and $\brac{\xi,\beta,\iota,\rho}$ be the unique solutions of the fluid model equations (for $K=1$)
corresponding to $\brac{\alpha^\circ,\mu}$ and $\brac{\alpha,\mu}$, respectively.
Then for all $x\in[0,\tau]$, $t\in[0,\tau]$ one has
$\rho\brac{t}=\rho^\circ\brac{t}$, $\beta_t^r\closedbrac{0,x}=\beta_t^{\circ,r}\closedbrac{0,x}$, $\beta_t^s\closedbrac{0,x}=\beta_t^{\circ,s}\closedbrac{0,x}$, and $\xi_t\closedbrac{0,x}=\xi_t^\circ\closedbrac{0,x}$.
\end{lemma}
\begin{proof}
Recall from the proof of Theorem \ref{thm:MD-MVSM} the notation
$\Gamma=(\Gamma_1,\Gamma_2)$ for the SM in the finite dimensional
orthant. We denote by $\Gamma^{(1)}=(\Gamma_1^{(1)},\Gamma_2^{(1)})$
the special case where the dimension is 1. That is, $\Gamma^{(1)}$
is merely the SM on the half line.
The proof uses crucially the non-anticipation property of this SM
\cite[p.165]{chen2001}, stating that if $\ph_i=\Gamma^{(1)}(\psi_i)$ for
$i=1,2$, and for some $T>0$, $\psi_1=\psi_2$ on $[0,T]$, then also
$\ph_1=\ph_2$ on $[0,T]$.

Next, recall the notation $\Theta$ for the VMVSM.
Once again, in the special case $K=1$ we denote this map by $\Theta^{(1)}$.
In this proof, $t$ is always assumed to be in $\closedbrac{0,\tau}$.
From Lemma \ref{thm:4.10 in atar18},
using equations $\eqref{107}$-\eqref{101} to reach conditions 1-4 in Definition \ref{def:MVSP} with the data $\brac{\alpha,\mu+\rho}$, it follows that
$\brac{\xi^\circ,\beta^\circ,\iota^\circ,\rho^\circ}$ and $\brac{\xi,\beta,\iota,\rho}$ are respectively the unique solutions of the following two problems \eqref{eq:Full hard fluid 0} and \eqref{eq:Full hard fluid}:
%
    \begin{align}\label{eq:Full hard fluid 0}\tag{P0}
        \begin{cases}
            (i) &\brac{\xi^\circ,\beta^\circ,\iota^\circ}=\Theta^{(1)}\brac{\alpha^\circ,\mu+\rho^\circ}, \\
            (ii) &\xi^\circ_t\left[0,t\right)=0,\\
            (iii) &\int\mathbbm{1}_{\{\sigma^\circ\brac{t}>t\}}d\rho^\circ\brac{t}=0,
            \quad\sigma^\circ\brac{t}=\inf\text{supp}\closedbrac{\xi^\circ_t}.
        \end{cases}
    \end{align}
    \begin{align}\label{eq:Full hard fluid}\tag{P1}
        \begin{cases}
             (i) &\brac{\xi,\beta,\iota}=\Theta^{(1)}\brac{\alpha,\mu+\rho}, \\
            (ii) &\xi_t\left[0,t\right)=0,\\
            (iii) &\int\mathbbm{1}_{\{\sigma\brac{t}>t\}}d\rho\brac{t}=0,
            \quad\sigma\brac{t}=\inf\text{supp}\closedbrac{\xi_t}.
        \end{cases}
    \end{align}
    Consider now the data $\brac{\alpha,\mu+\rho^\circ}\in\mathbb{D}_{\cal M}^\uparrow\times\mathbb{D}^\uparrow$, and denote the associated unique solution $\brac{\hat{\xi},\hat{\beta},\hat{\iota}}=\Theta^{(1)}\brac{\alpha,\mu+\rho^\circ}$, and $\hat{\sigma}\brac{t}=\inf\text{supp}\closedbrac{\hat{\xi}_t}$. 
    We will show that the tuple $\brac{\hat{\xi},\hat{\beta},\hat{\iota},\rho^\circ}$
    satisfies $(i)-(iii)$ in \eqref{eq:Full hard fluid}. Uniqueness implies then $\brac{\xi,\beta,\iota,\rho}=\brac{\hat{\xi},\hat{\beta},\hat{\iota},\rho^\circ}$.
    
    $\brac{\hat{\xi},\hat{\beta},\hat{\iota},\rho^\circ}$ satisfies condition (\ref{eq:Full hard fluid}.i) by the definition of $\brac{\hat{\xi},\hat{\beta},\hat{\iota}}$.
    
    To show that $\brac{\hat{\xi},\hat{\beta},\hat{\iota},\rho^\circ}$ satisfies (\ref{eq:Full hard fluid}.ii), use \cite[Lemma 2.7]{atar2018} for $x\leq\tau$:
    \begin{align*}
        \hat{\xi}\closedbrac{0,x}=\Gamma^{(1)}_1\brac{\alpha\closedbrac{0,x}-\mu-\rho^\circ}=\Gamma^{(1)}_1\brac{\alpha^\circ\closedbrac{0,x}-\mu-\rho^\circ}=\xi^\circ\closedbrac{0,x}.
    \end{align*}
    This implies immediately $\hat{\xi}_t\left[0,t\right)=\xi_t^\circ\left[0,t\right)=0$ by (\ref{eq:Full hard fluid 0}.ii).
    
    We need to show now that $\int\mathbb{I}_{\{\hat{\sigma}\brac{t}>t\}}d\rho^\circ\brac{t}=0$. This will follow from (\ref{eq:Full hard fluid 0}.iii) once we show that
    \begin{align*}
        \{t\leq\tau:\quad\sigma^\circ\brac{t}=t\}=\{t\leq\tau:\quad\hat{\sigma}\brac{t}=t\}.
    \end{align*}
    We already showed that for $x$ and $t$ in $\closedbrac{0,\tau}$: $\hat{\xi}_t\closedbrac{0,x}=\xi_t^\circ\closedbrac{0,x}$. Consider any
    $t^\prime\in\{t\leq\tau:\sigma^\circ\brac{t}=t\}$. Then $t^\prime=\inf\text{supp}\closedbrac{\xi^\circ_{t^\prime}}=\inf\text{supp}\closedbrac{\hat{\xi}_{t^\prime}}$. So $t^\prime$ is in $\{t\leq\tau:\hat{\sigma}\brac{t}=t\}$ and therefore
    \begin{align*}
        \{t\leq\tau:\quad\sigma^\circ\brac{t}=t\}\subseteq\{t\leq\tau:\quad\hat{\sigma}\brac{t}=t\}.
    \end{align*}
    The other direction is proved similarly.
    
    To conclude, when we consider only times in the interval $\closedbrac{0,\tau}$, the tuple $\brac{\hat{\xi},\hat{\beta},\hat{\iota},\rho^\circ}$ satisfies all the conditions in \eqref{eq:Full hard fluid}. So, by uniqueness, $\rho\brac{t}=\rho^\circ\brac{t}$ for $t\leq\tau$.
    
    This implies all the other equalities: $\beta_t^r\closedbrac{0,x}=\rho\brac{t\land x}=\rho^\circ\brac{t\land x}=\beta_t^\circ\closedbrac{0,x}$ for $x,t\in\closedbrac{0,\tau}$.
    
    Also, for $x,t\in\closedbrac{0,\tau}$, by the non-anticipation property,
    \begin{align*}
        \xi_t\closedbrac{0,x}=\Gamma^{(1)}_1\brac{\alpha\closedbrac{0,x}-\mu-\rho}\brac{t}=\Gamma^{(1)}_1\brac{\alpha^\circ\closedbrac{0,x}-\mu-\rho^\circ}\brac{t}=\xi_t^\circ\closedbrac{0,x}.
    \end{align*}
Finally, for $x,t\in\closedbrac{0,\tau}$: $\beta_t\closedbrac{0,x}=\alpha_t\closedbrac{0,x}-\xi_t\closedbrac{0,x}=\alpha_t^\circ\closedbrac{0,x}-\xi_t^\circ\closedbrac{0,x}=\beta_t^\circ\closedbrac{0,x}$, and $\beta_t^s\closedbrac{0,x}=\beta_t\closedbrac{0,x}-\beta_t^r\closedbrac{0,x}=\beta_t^\circ\closedbrac{0,x}-\beta_t^{\circ,r}\closedbrac{0,x}=\beta_t^{\circ,s}\closedbrac{0,x}$.
\end{proof}

The next lemma states that the fluid model equations for $K=1$
`preserve' Assumption \ref{ass:data assumptions}.

\begin{lemma}\label{lem:beta satisfies 1}
Let $K=1$ and $P=0$ and
suppose $\brac{\alpha,\mu}$ satisfies Assumption \ref{ass:data assumptions}.
Let $\brac{\xi,\beta,\iota,\rho}$ be the unique solution of the fluid model equations. Let $\gamma_t$
be defined via the relation $\gamma_t\closedbrac{0,x}=\beta_t^s\closedbrac{0,x-\epsilon}$. Then $\gamma$ also satisfies Assumption \ref{ass:data assumptions}.1 with zero initial condition,
i.e. $\gamma\in\mathbb{C}_{{\cal M}_\sim}^\uparrow$
    \begin{align}\label{111}
        \gamma_t(B)=\int_0^tg_s(B)ds,\quad t\geq 0,B\in\calB(\R_+),
    \end{align}
where for each $t$, $g_t$ is a finite measure on $\R_+$ with $g_t[0,t)=0$,
and the mapping $t\mapsto g_t(B)$ is measurable. Moreover,
$\lim_{\del\downarrow0}\sup_{s\in[0,t]}g_s[s,s+\del]=0$.
\end{lemma}
\begin{proof}
The proof is based mainly on the following two facts
\begin{equation}\label{105}
\beta^s_t(B)\le\al_t(B),\qquad B\in\calB(\R),
\end{equation}
\begin{equation}\label{106}
\beta^s_t(\R_+)-\beta^s_\tau(\R_+)\le\int_\tau^tm(s)ds,
\qquad 0\le\tau<t,
\end{equation}
and uses disintegration.

By the fluid model equations for $K=1$, using \eqref{eq:Hard fluid dynamics} (recalling
$P=0$) and \eqref{107},
for a Borel set $B$ we have $\alpha_t\brac{B}=\xi_t\brac{B}+\beta_t\brac{B}$
as well as $\beta_t(B)=\beta^s_t(B)+\beta^r_t(B)$, proving \eqref{105}.
Next, by \eqref{108}, the monotonicity of $\iota$ and Assumption \ref{ass:data assumptions}.2,
\eqref{106} holds.

    We first show that $\gamma$
belongs to $\mathbb{C}_{{\cal M}_\sim}^\uparrow$. From the axioms of our model $\beta^s\in\mathbb{D}_{{\cal M}}^\uparrow$. By \eqref{105} and the assumption
$\al_t\in\calM_\sim$ it follows that $\beta^s_t\in\calM_\sim$ for every $t$.
The continuity $t\mapsto\beta^s_t$ follows from \eqref{106}. This shows that $\beta^s$,
and in turn, $\gamma$, belongs to $\mathbb{C}_{{\cal M}_\sim}^\uparrow$.

    Consider now the space $\mathbb{R}_+^2$ with its Borel $\sigma$-algebra.
Let $\la$  be the measure on this space determined by $\gamma$ via the relation
\[
\lambda\brac{\closedbrac{x,y}\times\closedbrac{s,t}}=\gamma_{t}\closedbrac{x,y}-\gamma_s\closedbrac{x,y},
\qquad x<y,\ s<t.
\]
By \eqref{106}, the measure $\la(\R_+\times dt)$ is dominated by the measure
$m(t)dt$.
Let $\mathbb{T}:\brac{\mathbb{R}_+^2,{\cal B}(\R_+^2)}\to\brac{\R_+,{\cal B}\brac{\R_+}}$
be defined by $\mathbb{T}(x,t)=t$.
We use the disintegration theorem \cite[Theorem 1]{Chang1997},
with $\mathbb{T}$ as the measurable map.
According to this theorem there exists a family of finite measures $\tilde g_t$ on $\R_+$
such that $t\mapsto\tilde g_t(B)$ is measurable for every $B\in\calB(\R_+)$, and
for each nonnegative measurable $f$,
\begin{equation}\label{110}
\int f(x,t)\la(dx,dt)=\int f(x,t)\tilde g_t(dx)m(t)dt.
\end{equation}
Denote $\hat g_t(B)=\tilde g_t(B)m(t)$.
It will be shown that
\begin{equation}\label{109}
\hat g_t[0,t+\eps)=0 \quad \text{ for a.e. } t.
\end{equation}
Once the above is established, we can set
$g_t(B)=\hat g_t(B\cap[t+\eps,\iy))$, $B\in\calB(\R_+)$,
by which we achieve for a.e.\ $t$, $g_t(B)=\hat g_t(B)$.
Hence in view of \eqref{110} one has \eqref{111}.
As for the two final assertions made in the lemma,
we certainly have $g_t[0,t)\le g_t[0,t+\eps)=0$,
and for all small $\del$, $\sup_sg_s[s,s+\del]=0$,
by which these assertions are true.

We thus turn to showing \eqref{109}.
Fix $x\leq t\leq t_0$. From the hard EDF equations
and the assumption that jobs arrive before their deadlines:
    \begin{align}
        &0=\xi_t\closedbrac{0,x}=\alpha_t\closedbrac{0,x}-\beta_t\closedbrac{0,x},\nonumber\\
        &\alpha_{t_0}\closedbrac{0,x}-\alpha_{t}\closedbrac{0,x}=0,\nonumber\\
        &\beta_{t_0}^r\closedbrac{0,x}-\beta_{t}^r\closedbrac{0,x}=\rho\brac{x\land t_0}-\rho\brac{x\land t}=\rho\brac{x}-\rho\brac{x}=0,\label{eq:beta with x<t<t0}
    \end{align}
    \begin{align*}
        \Rightarrow\beta_{t_0}^s\closedbrac{0,x}-\beta_t^s\closedbrac{0,x}=\alpha_{t_0}\closedbrac{0,x}-\alpha_{t}\closedbrac{0,x}-\beta_{t_0}^r\closedbrac{0,x}+\beta_{t}^r\closedbrac{0,x}=0.
    \end{align*}
    This, together with \eqref{100} that expresses the fact that
deadlines are postponed by $\eps$ after service, implies that every rectangular subset $\closedbrac{0,t+\eps}\times\closedbrac{t,t_0}$ of the set $\{\brac{x,t}\in\mathbb{R}_+^2:x<t+\eps\}$ satisfies:
    \begin{align}\label{eq:hard fluid beta zero}
        \lambda\brac{\closedbrac{0,t+\eps}\times\closedbrac{t,t_0}}&=\gamma_{t_0}\closedbrac{0,t+\eps}-\gamma_t\closedbrac{0,t+\eps}\\
        &=\beta_{t_0}^s\closedbrac{0,t}-\beta_t^s\closedbrac{0,t}\nonumber\\
        &=0\nonumber.
    \end{align}
    And, since $\{\brac{x,t}\in\mathbb{R}_+^2:x<t+\epsilon\}$ is contained
    in a countable union of rectangles of this form,
    
    \begin{align*}
        0=\lambda\brac{\{\brac{x,t}\in\mathbb{R}_+^2:x<t+\eps\}}=\int_0^\infty\hat{g}_s\left[0,s+\eps\right)ds,
    \end{align*}
       and  \eqref{109} follows.
\end{proof}
\begin{remark}
Note that for any finite collection $\{\gamma^i\}_{i=1}^K$ and a set of positive coefficients $\{P_i\}_{i=1}^K$, if each element of the collection satisfies Assumption \ref{ass:data assumptions}, then does also $\sum_{i=1}^KP_i\gamma^i$. Lemma \ref{lem:beta satisfies 1} considers a single server, but we will use the conclusion on the sum when we will show that the total arrival  process, the sum of exogenous and the endogenous arrival processes, satisfies the assumptions.  
\end{remark}

\subsection{Proof of main results}\label{sec33}

This section opens with the proof of Theorem \ref{thm:uniqueness hard}
regarding existence and uniqueness of solutions to the fluid model equations.

\begin{proof}[Proof of Theorem \ref{thm:uniqueness hard}]
    Denote
    \begin{align*}
        &\Xi=\brac{\xi,\beta,\beta^s,\beta^r,\gamma,\iota,\rho},\\
        &\mathbb{X}=\mathbb{D}_{{\cal M}}^K\times\mathbb{D}_{{\cal M}}^{\uparrow K}\times\MVDupK\times\MVDupK\times\MVDupK\times\mathbb{D}^{\uparrow K}\times\mathbb{D}^{\uparrow K},
    \end{align*}
    \begin{align*}
        \calig{X}=\left\{
        \Xi\in\mathbb{X}:\quad\text{The components of $\Xi$ satisfy \eqref{100}-\eqref{eq:renege after deadline},\eqref{eq:hard fluid work cons}-\eqref{103}}
        \right\}.
    \end{align*}
Recall that the fluid model equations are \eqref{100}-\eqref{103}.
Equation \eqref{eq:Hard fluid dynamics}, the only one excluded from the
definition of $\calX$, is the equation that couples between the different servers.
Indeed, the remaining equations are fluid model equations for
$K$ separate single server systems. 
This exclusion makes it convenient
to use the single server results in our proof.
The existence proof, provided first, will be complete
once we construct a tuple $\Xi\in\calX$ that satisfies \eqref{eq:Hard fluid dynamics}.
It is then shown that this tuple is unique. Both existence and uniqueness are proved by induction.
    
    First consider, for each $i$, the fluid model solution of a single server with primitives $\brac{\alpha^i\brac{\cdot\cap\left[0,\eps\right)},\mu^i}$, denoted  $\brac{\xi^{i,(1)},\beta^{i,(1)},\iota^{i,(1)},\rho^{i,(1)}}$. 
    Denote also $\gamma_t^{i,(1)}\closedbrac{0,x}=\beta_t^{is,(1)}\closedbrac{0,x-\eps}$.
    Then, for $n>1$, denote the fluid model solution of a single server with primitives $\brac{\alpha^i\brac{\cdot\cap\left[0,n\eps\right)}+\sum_{j=1}^KP_{ji}\gamma^{j,(n-1)}\brac{\cdot\cap\lt[0,n\eps\rt)},\mu^i}$ by $\brac{\xi^{i,(n)},\beta^{i,(n)},\iota^{i,(n)},\rho^{i,(n)}}$ and $\gamma_t^{i,(n)}\closedbrac{0,x}=\beta_t^{i,(n)}\closedbrac{0,x-\eps}$. This inductively defines the
tuples $\Xi^{(n)}$.
Note that these tuples belong to $\calX$ as solutions to the fluid model of hard EDF servers. 
    Note also that all $\beta^{is,(n)}$ satisfy Assumption \ref{ass:data assumptions}.
    We now show that these tuples are consistent with each other in that for $\brac{x,t}\in\lt[0,n\eps\rt)^2$ and $m>n$, we have $\beta_t^{is,(n)}\closedbrac{0,x}=\beta_t^{is,(m)}\closedbrac{0,x}$, $\rho^{i,(n)}\brac{t}=\rho^{i,(m)}\brac{t}$ and $\xi_t^{i,(n)}\closedbrac{0,x}=\xi_t^{i,(m)}\closedbrac{0,x}$.

It is proved by induction over $n$ that these identities hold for all $m>n$.
    If $\brac{x,t}\in\lt[0,\eps\rt)^2$, then $\gamma_t^{j,(m)}\brac{\closedbrac{0,x}\cap\lt[0,m\eps\rt)}=\gamma_t^{j,(m)}\closedbrac{0,x}=0$ and $\alpha_t^{i}\brac{\closedbrac{0,x}\cap\lt[0,m\eps\rt)}=\alpha_t^{i}\brac{\closedbrac{0,x}\cap\lt[0,\eps\rt)}$, so the data generating $\rho^{(1)}$, $\rho^{(m)}$, $\xi^{(1)}$, $\xi^{(m)}$, $\beta^{s,(1)}$ and $\beta^{s,(m)}$ coincide on $\lt[0,\eps\rt)^2$ and Lemma \ref{lem:rho nonanticipation} yields the desired conclusion for $n=1$.
    Assuming the claim is true for $n$, consider $m\geq n+1$ and $\brac{x,t}\in\lt[0,\brac{n+1}\eps\rt)^2$. Then
    \begin{align*}
        \alpha^i_t\brac{\closedbrac{0,x}\cap\lt[0,\brac{n+1}\eps\rt)}=\alpha^i_t\brac{\closedbrac{0,x}\cap\lt[0,m\eps\rt)}=\alpha^i_t\closedbrac{0,x}.
    \end{align*}
    If $t<n\eps$, then the induction assumption implies 
    \begin{align*}
        \gamma_t^{j,(n)}\brac{\closedbrac{0,x}\cap\lt[0,\brac{n+1}\eps\rt)}=\beta_t^{js,(n)}\closedbrac{0,x-\eps}=\beta_t^{js,(m)}\closedbrac{0,x-\eps}=\gamma_t^{j,(m)}\brac{\closedbrac{0,x}\cap\lt[0,\brac{m+1}\eps\rt)}.
    \end{align*}
    If $t\in\left[n\eps,\brac{n+1}\eps\right)$, 
    by \eqref{eq:beta with x<t<t0}, we have that whenever  $t>x$: $\beta_t^{i,s}\closedbrac{0,x}=\beta_x^{i,s}\closedbrac{0,x}$. 
    It follows that if $t\in\left[n\epsilon,\brac{n+1}\epsilon\right)$,
    \begin{multline*}
        \gamma_t^{j,(n)}\brac{\closedbrac{0,x}\cap\lt[0,\brac{n+1}\eps\rt)}=\beta_t^{js,(n)}\closedbrac{0,x-\epsilon}=\beta_{\closedbrac{x-\epsilon}^+}^{js,(n)}\closedbrac{0,x-\epsilon}\\
        =\beta_{\closedbrac{x-\epsilon}^+}^{js,(m)}\closedbrac{0,x-\epsilon}=\gamma_t^{j,(m)}\closedbrac{0,x}=\gamma_t^{j,(m)}\brac{\closedbrac{0,x}\cap\lt[0,\brac{m+1}\eps\rt)}.
    \end{multline*}
    So, the data generating $\rho^{(n)}$, $\rho^{(m)}$, $\xi^{(n)}$, $\xi^{(m)}$, $\beta^{s,(n)}$ and $\beta^{s,(m)}$ coincide on $\lt[0,\brac{n+1}\eps\rt)^2$ and Lemma \ref{lem:rho nonanticipation} yields the desired conclusion for $n+1$.
    
The above shows that, for example, $\beta_t^{i,(n)}\closedbrac{0,x}$
becomes constant as $n$ increases. Hence we may extract from the
sequence a tuple in the following way.
Given $x$ and $t$, let $n>\brac{x\lor t}/\eps$ and let $\beta_t^i\closedbrac{0,x}=\beta_t^{i,(n)}\closedbrac{0,x}$, and $\xi_t^i\closedbrac{0,x}=\xi_t^{i,(n)}\closedbrac{0,x}$. 
For every fixed $t$, this uniquely defines measures by determining their
values on all sets $[0,x]$, a collection generating the $\sig$-algebra
$\calB(\R_+)$. 
Similarly, for each $t$, $\rho^{(n)}_t$ becomes constant as $n$ gets large,
hence we let $\rho=\lim_{n}\rho^{(n)}$.
We also define $\beta_t^{i,r}\closedbrac{0,x}=\rho^i\brac{x\land t}$, $\beta^{i,s}=\beta^i-\beta^{i,r}$, $\sigma^{i}\brac{t}=\inf\supp\closedbrac{\xi_t^i}$, $\iota^i=\mu^i-\beta^{i,s}\lt[0,\infty\rt)$ and $\gamma_t^i\closedbrac{0,x}=\beta_t^{i,s}\closedbrac{0,x-\eps}$.

    The tuple $\brac{\xi,\beta,\iota,\rho}$ thus constructed is our candidate, and we now show that it is indeed a solution to the fluid model equations.

First, by construction,
$\Xi$ satisfies \eqref{eq:Hard fluid dynamics}. 
\eqref{100}, \eqref{107}, \eqref{108}, \eqref{eq:rho def} and $\eqref{101}$ hold by definition, and \eqref{eq:hard fluid hardness} is immediate by construction.
\eqref{eq:renege after deadline} is also immediate by $\beta_t^{i,r}\brac{t,\infty}=\rho^i\brac{t}-\rho^i\brac{t\land t}=0$.
Equations \eqref{eq:hard fluid work cons} and \eqref{eq:hard fluid EDF cond} are satisfied because, as will soon be shown, $\psi_x^{i,(n)}:=\beta^{is,(n)}\brac{x,\infty}+\iota^{i,(n)}-\beta^{is}\brac{x,\infty}-\iota^i\in\Dup$ and $\xi_t^{i,(n)}\closedbrac{0,x}=\xi_t^i\closedbrac{0,x}$ for large enough $n$.
Recall also that $\Xi^{(n)}\in\calX$ and \eqref{eq:hard prelimit work cons} and \eqref{eq:hard prelimit EDF cond} .
Therefore,
\begin{align}
    &\int\xi_s^i\closedbrac{0,x}d\beta_s^{i,s}\brac{x,\infty}+\int\xi_s^i\closedbrac{0,x}d\iota^i\brac{s}\nonumber\\
    &=\lim_{n\to\infty}\abs{\int\xi_s^{i,(n)}\closedbrac{0,x}d\beta_s^{is,(n)}\brac{x,\infty}+\int\xi_s^{i,(n)}\closedbrac{0,x}d\iota^{i,(n)}\brac{s}-\int\xi_s^{i}\closedbrac{0,x}d\beta_s^{is}\brac{x,\infty}-\int\xi_s^i\closedbrac{0,x}d\iota^i\brac{s}}\nonumber\\
    &=\lim_{n\to\infty}\abs{\int\xi_s^i\closedbrac{0,x}d\psi_x^{i,(n)}\brac{s}}\nonumber\\
    &\leq\norm{\xi^i\closedbrac{0,x}}_T\lim_{n\to\infty}\brac{\beta_T^{is,(n)}\brac{x,\infty}+\iota^{i,(n)}\brac{T}-\beta_T^{is}\brac{x,\infty}-\iota^i\brac{T}}=0.\label{eq:lim of beta R+}
\end{align}
For the equality in \eqref{eq:lim of beta R+}, note that $\beta_T^{i,s}\lt[0,\infty\rt)=\lim_m\lim_n\beta_T^{is,(n)}\closedbrac{0,m}$ while $\beta_T^{is,(n)}\closedbrac{0,m}$ is monotone in both $n$ and $m$, hence interchanging the limits is justified.
Thus we have $\beta_T^{is,(n)}\lt[0,\infty\rt)\to \beta_T^{i,s}\lt[0,\infty\rt)$ and $\iota^{i,(n)}\brac{T}\to\iota^i\brac{T}$ and \eqref{eq:lim of beta R+}.

Lemma 2.2 in \cite{atar2018} implies $\psi_x^{i,(n)}\in\Dup$ and the monotonicity in $n$ of $\beta_t^{is,(n)}\closedbrac{0,x}$, provided one shows that for all $n$
    \begin{align*}
        \alpha^i\lt[x\land n\eps,x\land \brac{n+1}\eps\rt)+\sum_{j=1}^KP_{ji}\brac{\gamma^{j,(n)}\lt[0,x\land \brac{n+1}\eps\rt)-\gamma^{j,(n-1)}\lt[0,x\land n\eps\rt)}\in\Dup.
    \end{align*}
The first term is in $\Dup$ by assumption, and the second term is now shown to belong to $\Dup$ by induction, where each step invokes Lemma 2.2 from \cite{atar2018}.
    Indeed, the lemma yields
    \begin{align*}
        \beta^{is,(2)}\lt[0,x\land 2\eps\rt)-\beta^{is,(1)}\lt[0,x\land \eps\rt)=\mu^i-\iota^{i,(2)}-\beta^{is,(2)}\lt[x\land 2\eps,\infty\rt)-\mu^i+\iota^{i,(1)}+\beta^{is,(1)}\lt[x\land \eps,\infty\rt)\in\Dup
    \end{align*}
    because $\alpha^{i}\lt[x\land\eps,x\land 2\eps\rt)+\sum_{j=1}^KP_{ji}\gamma^{j,(1)}\lt[0,x\land 2\eps\rt)\in\Dup$. 
    Also, 
    \begin{multline*}
        \beta^{is,(n)}\lt[0,x\land n\eps\rt)-\beta^{is,(n-1)}\lt[0,x\land\brac{n-1}\eps\rt)\\
        =\mu^i-\iota^{i,(n)}-\beta^{is,(n)}\lt[x\land n\eps,\infty\rt)-\mu^i+\iota^{i,(n-1)}+\beta^{is,(n-1)}\lt[x\land\brac{n-1}\eps,\infty\rt)\in\Dup
    \end{multline*}
    if $\alpha^{i}\lt[x\land\brac{n-1}\eps,x\land n\eps\rt)+\sum_{j=1}^KP_{ji}\brac{\gamma^{j,(n-1)}\lt[0,x\land n\eps\rt)-\gamma^{j,(n-2)}\lt[0,x\land\brac{n-1}\eps\rt)}\in\Dup$.

    For \eqref{103}, note that 
    \begin{align*}
        \{t:\sigma^{i}\brac{t}>t\}\subset\bigcup_{n\geq m}\{t:\sigma^{i,(n)}>t\}\quad \forall m.
    \end{align*}
    This is due to the fact that if $\xi_t^i\closedbrac{0,x}>0$ then also $\xi_t^{i,(n)}\closedbrac{0,x}>0$ for all $n>\brac{x\lor t}/\eps$.
Consider $t\in[0,T]$ for $T$ fixed.
Recall that $\rho^i_t=\rho^{i,(n)}_t$ for all $t\in[0,T]$, $n>T/\eps$.
Using the above display and the union bound, we obtain for all large $n$,
    \begin{align*}
        \int_{[0,T]}\indicator{\sigma^i\brac{s}>s}d\rho^i\brac{s}\leq\sum_{n>T/\eps}\int\indicator{\sigma^{i,(n)}\brac{s}>s}d\rho^i\brac{s}=\sum_{n>T/\eps}\int\indicator{\sigma^{i,(n)}\brac{s}>s}d\rho^{i,(n)}\brac{s}=0.
    \end{align*}
This shows $\Xi\in\calX$ and completes
the existence proof.    
    
    The argument for uniqueness is as follows.
    We show that for every $x,t\in\mathbb{R_+}$, the quantities $\xi_t\closedbrac{0,x}$, $\beta_t\closedbrac{0,x}$ $,\rho\brac{t}$ and $\iota\brac{t}$ are uniquely determined by the primitives $\brac{\alpha,\mu}\in\MVDupK\times\DupK$. This we do by arguing that the claim
holds for every $\brac{x,t}\in\left[0,n\eps\right)^2$, by induction in $n$.
%
    
    First, directly from \eqref{eq:Hard fluid dynamics}, for $x\in\left[0,\eps\right)$, for the $i$-th server:
    \begin{align*}
        \xi_t^i\closedbrac{0,x}&=\alpha_t^i\closedbrac{0,x}-\beta_t^i\closedbrac{0,x}.
    \end{align*}
    In addition, $\brac{\xi^i,\beta^i,\iota^i,\rho^i}\in\calig{X}$.
    By Lemma \ref{lem:rho nonanticipation}, $\xi_t^i\closedbrac{0,x}$, $\beta_t^{is}\closedbrac{0,x}$ and $\rho^i\brac{t}$ coincide with the unique solution to the fluid model equations of a single server with primitives $\brac{\alpha\brac{\cdot\cap\left[0,\eps\right)},\mu}$ on $\left[0,\eps\right)^2$. In particular, for $t,x\in\left[0,\eps\right)$, $\beta_t^{is}\closedbrac{0,x}$ can be written as required in Assumption \ref{ass:data assumptions}.1. 
    
Next, assume that the uniqueness statement holds
for $(x,t)\in\left[0,n\eps\right)^2$.
Let $\brac{\xi^i,\beta^i,\iota^i,\rho^i}$ denote 
the unique tuple satisfying for all $1\leq i\leq K$ and $(x,t)
\in\left[0,n\eps\right)^2$,
    \begin{align*}
        &\xi_t^i\closedbrac{0,x}=\alpha_t^i\closedbrac{0,x}+\sum_{j=1}^KP_{ji}\beta_t^{j,s}\closedbrac{0,x-\eps}-\beta_t^i\closedbrac{0,x},\\
        &\brac{\xi^i,\beta^i,\iota^i,\rho^i}\in\calig{X}.
    \end{align*}
    Assume in addition that for those $x$ and $t$, $\beta_t^{is}\closedbrac{0,x}$ can be written as required in Assumption \ref{ass:data assumptions}.1.
    
    Consider now $(x,t)\in\left[0,\brac{n+1}\eps\right)^2$. First, directly from Equation \eqref{eq:Hard fluid dynamics}, for the $i$-th server:
    \begin{align*}
        &\xi_t^i\closedbrac{0,x}=\alpha_t^i\closedbrac{0,x}+\sum_{j=1}^KP_{ji}\beta_t^{j,s}\closedbrac{0,x-\eps}-\beta_t^i\closedbrac{0,x},\\
        &\brac{\xi^i,\beta^i,\iota^i,\rho^i}\in\calig{X}.
    \end{align*}
    If $x$ is in $\left[0,\brac{n+1}\eps\right)$ then $x-\eps<n\eps$. 
    Therefore, $\beta_t^{i,s}\closedbrac{0,x-\eps}$ is uniquely determined
    by $(\al,\mu)$ for $t\in\left[0,n\eps\right)$ by our induction assumption. 
    Now for $t\in\left[n\eps,\brac{n+1}\eps\right)$, 
by \eqref{eq:beta with x<t<t0}, we have that whenever  $t>x$: $\beta_t^{i,s}\closedbrac{0,x}=\beta_x^{i,s}\closedbrac{0,x}$. 
It follows that if $t\in\left[n\epsilon,\brac{n+1}\epsilon\right)$ then for $1\leq j\leq K$: $\beta_t^{j,s}\closedbrac{0,x-\epsilon}=\beta_{\closedbrac{x-\epsilon}^+}^{j,s}\closedbrac{0,x-\epsilon}$, which is uniquely determined, in view of the induction assumption. 
    
    By Lemma \ref{lem:rho nonanticipation}, $\brac{\xi^i,\beta^i,\rho^i}$ coincides
     on $\left[0,\brac{n+1}\epsilon\right)^2$ with the
    unique solution of the fluid model equations for a single server with primitives
$\brac{\alpha^i\brac{\cdot\cap\left[0,\brac{n+1}\eps\right)}+\sum_{j=1}^KP_{ji}\gamma^j\brac{\cdot\cap\left[0,\brac{n+1}\eps\right)},\mu}$. In particular, $\beta_t^{is}\closedbrac{0,x}$ can be written as required in Assumption \ref{ass:data assumptions}.1 for $x$ and $t$ in $\left[0,\brac{n+1}\eps\right)$ as well. 
    
    Finally, $\iota$ is uniquely determined for all $t$ via $\iota=\mu+\rho-\beta^s\left[0,\infty\right)$.
\end{proof}
Our final goal is to prove convergence to the fluid model equations.
We first show that the sequence $\{\bar{\rho}^N\}$ is tight,
and deduce from that tightness of the entire tuple appearing on
the LHS of \eqref{112}. Later we complete the proof by
showing that any subsequential limit satisfies the fluid model equations.


\begin{lemma}\label{lem:rho is tight}
    The sequence $\{\bar{\rho}^N\}$ is $C$-tight. Consequently, the tuple $\brac{\bar{\xi}^N,\bar{\beta}^{s,N},\bar{\beta}^{r,N},\bar{\beta}^N,\bar{\gamma}^N,\bar{\iota}^N}$ is C-tight.
\end{lemma}
\begin{proof}
The second statement follows from the first
by a simple inductive argument over the squares $\lt[0,n\eps\rt)^2$
based on the continuous mapping theorem and Remark 5.1 of \cite{atar2018}.
We omit the details.

To prove the first claim, fix $T$. 
$C$-tightness is shown for $\{\bar\rho^N|_{[0,T]}\}$.
Define $F_{\bar{\alpha}^{i,N}}\brac{x}=\bar{\alpha}_T^{i,N}\closedbrac{0,x}$. To show $C$-tightness, note first that $\|\bar\rho^N\|_T$ is dominated by $\|\bar\al^N\|_T$,
that is a tight sequence of RVs. Hence it remains to show that
for every $\del>0$, $w_T(\bar\rho^N,\del)\to0$ in probability, as $N\to\infty$.
To this end,
we bound the aforementioned modulus of continuity
in terms of $w_\iy\brac{F_{\bar{\alpha}_T^{i,N}},\del}$ using the following chain of inequalities.
The justification of each step in this chain is given below. For $0\le t\le t+\del\le T$,
\begin{align}
    \bar{\rho}^N\brac{t+\delta}-\bar{\rho}^N\brac{t}&\leq\bar{\alpha}^N_{t+\delta}\closedbrac{0,t+\delta}-\bar{\alpha}^N_{t}\closedbrac{0,t}+\sum_{j=1}^K\bar{\gamma}_{t+\delta}^{ji,N}\closedbrac{0,t+\delta}-\sum_{j=1}^K\bar{\gamma}_{t}^{ji,N}\closedbrac{0,t}\label{eq:rhotight1}\\
    &\leq\bar{\alpha}_T^{i,N}\closedbrac{t,t+\delta}+\sum_{j=1}^K\bar{\gamma}_{t+\delta}^{j,N}\left[t,t+\delta\right]\label{eq:rhotight2}\\
    &\leq\bar{\alpha}_T^{i,N}\closedbrac{t,t+\delta}+\sum_{j=1}^K\bar{\beta}_{t+\delta}^{js,N}\left[t-\eps,t+\delta-\eps\right]+N^{-1}\label{eq:rhotight3}\\
    &\leq\bar{\alpha}_T^{i,N}\closedbrac{t,t+\delta}+\sum_{n=1}^{\left\lfloor\frac{t+\delta}{\eps}\right\rfloor}\sum_{j=1}^K\brac{\bar{\alpha}_T^{j,N}\closedbrac{t-n\eps,t+\delta-n\eps}+\frac{1}{N}}+N^{-1}\label{eq:rhotight4}\\
    &\leq\brac{\frac{KT}{\eps}+1}\brac{\max_{i}w\brac{F_{\bar{\alpha}_T^{i,N}},\delta}+\frac{1}{N}}.\nonumber
\end{align}
The convergence $\bar\al^N\To\al$ and the continuity of the path $t\mapsto\al_t$ we have
$\bar\al^N_T\To\al_T$. Since by Assumption \ref{ass:data assumptions}, $\al_T$ has no atoms,
the continuous, monotone, bounded function $x\mapsto F_{\al_T}(x)$ is uniformly continuous,
and we have
$\lim_{\eta\downarrow0}\limsup_NP(w_\iy(F_{\bar\al_T^{i,N}},\del)>\eta)=0$.
This shows
that $\{\bar{\rho}^N\}$ are $C$-tight.

It remains to prove the chain of inequalities.

Inequality \eqref{eq:rhotight1} follows from \eqref{eq:hard Jackson dynamics} with $x=t$, $\bar{\xi}_t^{i,N}\closedbrac{0,t}=0$, and $\bar{\beta}_t^{is,N}\closedbrac{0,t}\leq\bar{\beta}_{t+\delta}^{is,N}\closedbrac{0,t+\delta}$.

For inequality \eqref{eq:rhotight2}, first, $\bar{\alpha}_{t+\delta}^{i,N}[0,t]=\bar{\alpha}_t^{i,N}\closedbrac{0,t}$ because, by assumption, no job enters the system with an overdue deadline
(see Figure \ref{fig:Illustration for tight rho}). Hence
\begin{align*}
    \bar{\alpha}^N_{t+\delta}\closedbrac{0,t+\delta}-\bar{\alpha}^N_{t}\closedbrac{0,t}&=\bar{\alpha}^N_{t+\delta}\closedbrac{0,t+\delta}-\bar{\alpha}^N_{t+\delta}\closedbrac{0,t}+\bar{\alpha}_{t+\delta}^{i,N}[0,t]-\bar{\alpha}_t^{i,N}\closedbrac{0,t}\\
    &\leq\bar{\alpha}^N_{t+\delta}\closedbrac{t,t+\delta}\\
    &\leq\bar{\alpha}^N_{T}\closedbrac{t,t+\delta}.
\end{align*}
Then, use a similar property for $\bar{\gamma}^{is,N}$:  $\bar{\gamma}_{t+\delta}^{is,N}\closedbrac{0,t}=\bar{\gamma}_{t}^{is,N}\closedbrac{0,t}$. 
This is due to the fact that no job is served after its deadline has expired.
To see this, fix any $x\leq t\leq t_0$, and recall that by \eqref{104}
\begin{align}\label{eq:hard prelimit gamma ij diff}
    \bar{\gamma}_{t_0}^{ji,N}\closedbrac{0,x}-\bar{\gamma}_{t}^{ji,N}\closedbrac{0,x}=\int_{\lt(t,t_0\rt]}\theta^{ji,N}\brac{s}d\bar{\gamma}_s^{i,N}\lt[0,x\rt]=0.
\end{align}
It follows that $\bar{\gamma}_{t+\delta}^{ji,N}\closedbrac{0,t+\delta}-\bar{\gamma}_{t}^{ji,N}\closedbrac{0,t+\delta}\leq\bar{\gamma}_{t+\delta}^{j,N}\closedbrac{t,t+\delta}$.

\begin{figure}
    \centering
    \begin{tikzpicture}
        \draw[<->] (0,2.5) node[anchor=east]{$t$} --(0,0)--(2.5,0) node[anchor=north]{$x$};
        \draw (0,0)--(2,2);
        \draw (1,0) rectangle (1.5,1.5);
        \draw (0,1) rectangle (1,1.5);
        \node[anchor=north east] at (1,0) {$t$};
        \node[anchor=north] at (1.5,0) {$t+\delta$};
        \node at (0.5,1.25) {0};
    \end{tikzpicture}
    \caption{Illustration for Inequality \eqref{eq:rhotight2}.}
    \label{fig:Illustration for tight rho}
\end{figure}
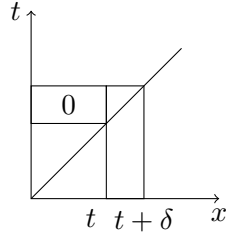

Inequality \eqref{eq:rhotight3} follows from Equation \eqref{eq:hard prelimit gamma beta}.

For \eqref{eq:rhotight4}, we prove by induction that
\begin{align*}
    \sum_{j=1}^K\bar{\beta}_{T}^{js,N}\left[t,t+\delta\right]\leq\sum_{n=0}^{\left\lfloor\frac{t+\delta}{\eps}\right\rfloor}\sum_{j=1}^K\brac{\bar{\alpha}_T^{j,N}\closedbrac{t-n\eps,t+\delta-n\eps}+\frac{1}{N}}.
\end{align*} 

If $t$ is such that $t+\delta<\eps$, then the statement follows from Equation \eqref{eq:hard Jackson dynamics} and $\bar{\beta}_T^{jr,N}\closedbrac{t,t+\delta}\geq 0$.

Now, assume that the desired inequality holds for $t$ such that
$\left\lfloor\frac{t+\delta}{\eps}\right\rfloor=m-1$.
Then for $t$ such that $\left\lfloor\frac{t+\delta}{\eps}\right\rfloor=m$, we have
\begin{align*}
    \sum_{j=1}^K\bar{\beta}_{T}^{js,N}\left[t-\eps,t+\delta-\eps\right]\leq\sum_{n=1}^{m}\sum_{j=1}^K\brac{\bar{\alpha}_T^{j,N}\closedbrac{t-n\eps,t+\delta-n\eps}+\frac{1}{N}}.
\end{align*}
For these values of $t$, from \eqref{eq:hard Jackson dynamics},
\begin{align*}
    \bar{\beta}_T^{is,N}\closedbrac{t,t+\delta}\leq\bar{\alpha}_T^{i,N}\closedbrac{t,t+\delta}+\sum_{j=1}^K\bar{\gamma}_T^{ji,N}\closedbrac{t,t+\delta}.
\end{align*}
Summing over all $j$, using \eqref{eq:hard prelimit gamma beta} and the assumption,
\begin{align*}
        \sum_{i=1}^K\bar{\beta}_T^{is,N}\closedbrac{t,t+\delta}&=\sum_{i=1}^K\bar{\alpha}_T^{i,N}\closedbrac{t,t+\delta}+\sum_{j=1}^K\bar{\gamma}_T^{j,N}\closedbrac{t,t+\delta}\\
        &\leq\sum_{i=1}^K\bar{\alpha}_T^{i,N}\closedbrac{t,t+\delta}+\sum_{j=1}^K\bar{\beta}_T^{js,N}\closedbrac{t-\eps,t+\delta-\eps}+\sum_{j=1}^K\frac{1}{N}\\
        &\leq\sum_{i=1}^K\bar{\alpha}_T^{i,N}\closedbrac{t,t+\delta}+\sum_{n=1}^{m}\sum_{j=1}^K\brac{\bar{\alpha}_T^{j,N}\closedbrac{t-n\eps,t+\delta-n\eps}+\frac{1}{N}}+\sum_{j=1}^K\frac{1}{N}\\
        &=\sum_{n=0}^{m}\sum_{j=1}^K\brac{\bar{\alpha}_T^{j,N}\closedbrac{t-n\eps,t+\delta-n\eps}+\frac{1}{N}}.
\end{align*}
This completes the proof of the chain of inequalities and the result follows.
\end{proof}

This means that the tuple is also sequentially compact. We will use that fact soon to prove Theorem \ref{thm:hard convergence}.

\begin{proof}[Proof of Theorem \ref{thm:hard convergence}]
First, the convergence $\brac{\bar{e}^N,\bar{E}^N}\Rightarrow\brac{0,0}$ follows by  the same reasons as for the soft version.
Next, from Lemma \ref{lem:rho is tight} and Prohorov's theorem,
every subsequence of the tuple has a convergent subsequence.
We will show that any subsequential limit must be a solution
to the fluid model; uniqueness implies then convergence of the entire sequence. 

Consider a convergent subsequence
and denote by $(\rho,\al,\mu,\xi,\beta^s,\beta^r,\beta,\gamma,\iota)$
its limit.
We appeal to Skorokhod representation theorem (\cite[Thm. 6.7]{Billingsley1999}), and assume without loss of generality convergence a.s.
It is now argued that the tuple $(\rho,\al,\mu,\xi,\beta^s,\beta^r,\beta,\gamma,\iota)$ satisfies the fluid model equations.


Equation \eqref{108} follows by $\bar{e}^{i,N}\To 0$. 
Identities \eqref{100}, \eqref{107}, \eqref{eq:Hard fluid dynamics}, and \eqref{eq:rho def}
will follow from the convergence of the tuple once it is shown that
for every $t$, the measures $\xi_t$, $\beta_t$, etc., have no atoms.
To show that these measures have no atoms, consider
relation \eqref{eq:hard Jackson dynamics} with $x<\eps$, in which case
the term involving $\gamma^N$ is absent.
Using the assumption that $\alpha\in\MVCupnoatomsK$
 and Portmanteau theorem, for every $0\leq a<b<\eps$, a.s.,
\begin{align*}
    \beta_t^{is}\brac{a,b}\leq\liminf\bar{\beta}_t^{is,N}\brac{a,b}\leq\liminf
    \bar{\alpha}_t^{i,N}\brac{a,b}=\al^i_t(a,b).
\end{align*}
Hence the fact that $\al^i_t$ has no atoms implies that the same is true for $\beta^{is}_t$,
on the interval $[0,\eps)$.
Hence the same holds for $\gamma^i_t$ and a similar argument holds for $\xi^i_t$.
An inductive argument over intervals $[0,n\eps)$ used in \eqref{eq:hard Jackson dynamics}
shows that these measures are all atomless on all of $\R_+$. (The induction argument
is omitted).


Moving now to show \eqref{eq:hard fluid work cons} and \eqref{eq:hard fluid EDF cond}, note that by  \eqref{eq:hard prelimit gamma beta}, \eqref{eq:hard Jackson dynamics},  \eqref{eq:hard prelimit work cons}, \eqref{eq:hard prelimit EDF cond}, and \eqref{eq:errors hard},
\begin{align*}
    &\brac{\bar{\xi}^{i,N}\closedbrac{0,x},\bar{\beta}^{is,N}\brac{x,\infty}+\bar{\iota}^{i,N}}=\\
    &\Gamma^{(1)}\Bigg(\bar{\alpha}^{i,N}\closedbrac{0,x}+\sum_{j=1}^KP_{ji}\brac{\bar{B}^{i,N}\closedbrac{0,x-\eps}-\bar{B}^{i,N}_{0-}\closedbrac{0,x-\eps}}\\
    &\quad+\sum_{j=1}^KP_{ji}\bar{\beta}^{js,N}\closedbrac{0,x-\eps}-\bar{\rho}^{i,N}\brac{\cdot\land x}-\bar{\mu}^{i,N}-\bar{e}^{i,N}+\sum_{j=1}^K\bar{E}^{ji,N}\brac{\cdot,x}\Bigg).
\end{align*}
Recalling that $\Gamma^{(1)}$ is continuous, using \eqref{eq:d dominations} and the fact that there are no atoms, one obtains \eqref{eq:hard fluid work cons} and \eqref{eq:hard fluid EDF cond} by the definition of the Skorokhod map.

It remains to show that the limit satisfies the condition $\int\mathbbm{1}_{\{\sigma^i\brac{t}>t\}}d\rho\brac{t}=0$. 
The idea is similar to the proof in Section 5.1.4 in \cite{atar2018};
however, there are many details that are different.
By Fatou's lemma, it is enough to prove  that the event
\begin{align*}
    E_0^i=\left\{\int_0^T\indicator{\sigma^i\brac{t}>t+\delta}d\rho^i\brac{t}>0\right\}
\end{align*}
occurs with probability zero for all $1\leq i\leq K$. We refer to Lemma 5.9 in \cite{atar2018} and note that there exists a $\lt[0,T\rt)\cup\{\infty\}$-valued random variable $\tau$ such that $\prob{E_0^i}=\prob{E_1^i\cap E_2^i}$ where
\begin{align*}
    &E_1^i=\left\{\tau<T,\sigma^i\brac{\tau}>\tau+\delta\right\},
    &E_2^i=\left\{\rho^i\brac{\tau+\delta}>\rho^i\brac{\tau},\forall\delta>0\right\}.
    \end{align*}
Define
    \begin{align*}
    &E_3^i=\left\{\exists\delta\brac{\omega}>0:\rho^i\brac{\tau+\delta}=\rho^i\brac{\tau}\right\}
\end{align*}
and note that $\prob{E_1^i\cap E_2^i}=\prob{E_1^i\cap E_3^{ic}}$.

We wish to show that for any $\delta>0$, $\prob{E_0^i}=0$, which is equivalent to showing $\prob{E_1^i\cap E_3^{ic}}=0$ for all $i$. Note that it is enough to show this for any $\eps>\delta>0$.
In fact, we shall take $\delta<\eps\land\delta_0$, where $\delta_0\in\brac{0,1}$ satisfies 
\begin{align}\label{eq:delta 0}
    a_s^i\closedbrac{s,s+2\delta_0}<m^i\brac{s}\text{ for all }s\in\closedbrac{0,T+1}\text{ and }i\in\{1,...,K\}.
\end{align}
The existence of such $\delta_0$ is guaranteed by Assumption \ref{ass:data assumptions}.

By \eqref{eq:hard Jackson dynamics} and \eqref{eq:hard prelimit hardness}, for $b>a$,
\begin{align*}
    \bar{\rho}^{i,N}\brac{b}-\bar{\rho}^{i,N}\brac{a}+\bar{\beta}_b^{is,N}\lt(a,b\rt]-\bar{\beta}_a^{is,N}\lt(a,b\rt]=\bar{\alpha}_b^{i,N}\lt(a,b\rt]-\bar{\alpha}_a^{i,N}\lt(a,b\rt]+\bar{\xi}_a^{i,N}\lt(a,b\rt]+\sum_{j=1}^K\brac{\bar{\gamma}_b^{ji,N}\lt(a,b\rt]-\bar{\gamma}_a^{ji,N}\lt(a,b\rt]}.
\end{align*}

Partition $\left(\tau,\tau+\delta\right]$ into $M\in\N$ subintervals  $I_m=\left(t_{m-1},t_m\right]$, with $\delta_M=M^{-1}\delta$ and $t_m=\tau+m\delta_M$, $m=1,...,M$, and bound the increment of $\bar{\rho}^{i,N}$ by
\begin{align*}
    \bar{\rho}^{i,N}\brac{\tau+\delta}-\bar{\rho}^{i,N}\brac{\tau}=\sum_{m=1}^M\brac{\bar{\rho}^{i,N}\brac{t_m}-\bar{\rho}^{i,N}\brac{t_{m-1}}}\leq C_{N,M}^i+D_{N,M}^i+G_{N,M}^i,
\end{align*}
where 
\begin{align*}
    &C_{N,M}^i=\sum_{m=1}^M\bar{\xi}_{t_{m-1}}^{i,N}\brac{I_m},\\
    &D_{N,M}^i=\sum_{m=1}^M\brac{\bar{\alpha}_{t_m}^{i,N}\brac{I_m}-\bar{\alpha}_{t_{m-1}}^{i,N}\brac{I_m}},\\
    &G_{N,M}^i=\sum_{m=1}^M\sum_{j=1}^K\brac{\bar{\gamma}_{t_m}^{ji,N}\brac{I_m}-\bar{\gamma}_{t_{m-1}}^{ji,N}\brac{I_{m}}}.
\end{align*}
We first fix $M$ and let $N\to\infty$, and then let $M\to \infty$ to obtain $\mathbbm{1}_{E_1^i}\brac{C_{N,M}^i+D_{N,M}^i+G_{N,M}^i}\to 0$ a.s., which implies $\mathbbm{1}_{E_1^i}\brac{\rho^i\brac{\tau+\delta}-\rho^i\brac{\tau}}=0$ a.s.
This completes the proof by concluding $\prob{E_1^i\cap E_3^{ic}}=0$ for all $i$.

By assumption, $D_{N,M}$ converges as $N\to\infty$ to
\begin{align*}
    D_M=\sum_{m=1}^M\int_{t_{m-1}}^{t_m}a_s^i\brac{I_m}ds\leq\brac{T+\delta}\sup_{s\in\closedbrac{0,T}}a_s^i\closedbrac{s,s+\delta_M},
\end{align*}
and, as $M\to\infty$, $\sup_{s\in\closedbrac{0,T}}a_s^i\closedbrac{s,s+\delta_M}\to 0$ as assumed in Assumption \ref{ass:data assumptions}.

Next, note that $C_{N,M}^i\leq M\max_{s\in\closedbrac{\tau,\tau+\delta}}\bar{\xi}_s^{i,N}\lt(\tau,\tau+\delta\rt]$, and recall that $\bar{\xi}^{i,N}\to\xi^i\in\MVCnoatoms$ a.s. 
Hence
\begin{align*}
    \sup_{s\in\closedbrac{0,T}}\sup_x\abs{\bar{\xi}_s^{i,N}\closedbrac{0,x}-\xi_s^i\closedbrac{0,x}}\to 0\text{ a.s.}
\end{align*}
This, together with $\mathbbm{1}_{E_1^i}\xi_\tau\closedbrac{\tau,\tau+\delta}=0$ and $\xi_\tau\lt[0,\tau\rt)=0$, implies $\mathbbm{1}_{E_1^i}\bar{\xi}_\tau^{i,N}\closedbrac{\tau,\tau+\delta}\to 0$ a.s.
By virtue of the shift property of the Skorokhod mapping, $\xi_{\tau+\cdot}^i\closedbrac{0,\tau+\delta}=\Gamma_1^{(1)}\brac{\psi^{i,\tau,\delta}}$, where:
\begin{align*}
    \psi^{i,\tau,\delta}\brac{t}=&\xi_\tau^i\closedbrac{0,\tau+\delta}+\alpha_{\tau+t}^i\closedbrac{0,\tau+\delta}-\alpha_\tau^i\closedbrac{0,\tau+\delta}\\
    &+\sum_{j=1}^KP_{ji}\brac{\beta_{\tau+t}^{js}\closedbrac{0,\tau+\delta-\eps}-\beta_{\tau}^{js}\closedbrac{0,\tau+\delta-\eps}}\\
    &-\beta_{\tau+t}^{ir}\closedbrac{0,\tau+\delta}-\beta_\tau^{ir}\closedbrac{0,\tau+\delta}-\mu^i\brac{\tau+t}+\mu^i\brac{\tau}.
\end{align*}
Notice that if $\delta$ is smaller than $\eps$ then the sum over $j$ is just 0, as in \eqref{eq:hard fluid beta zero} . 
Moreover,
\begin{align*}
    \alpha_{\tau+t}^i\closedbrac{0,\tau+\delta}-\alpha_\tau^i\closedbrac{0,\tau+\delta}-\mu^i\brac{\tau+t}+\mu^i\brac{\tau}=\int_\tau^{\tau+t}a_s^i\closedbrac{0,\tau+\delta}ds-\int_\tau^{\tau+t}m^i\brac{s}ds
\end{align*}
is non-increasing for $t\in\closedbrac{0,\delta_0}$ by \eqref{eq:delta 0}. 
Therefore, $\xi_t^{i,N}\closedbrac{0,\tau+\delta}=0$ for all $t\in\closedbrac{\tau,\tau+\delta}$, resulting with $\mathbbm{1}_{E_1^i}C_{N,M}\to 0$ a.s.

As for $G_{N,M}^i$, we start with the bound
\begin{align*}
    G_{N,M}^i\leq \sum_{m=1}^M\sum_{j=1}^K\brac{\bar{\gamma}_{t_{m}}^{j,N}\brac{I_m}-\bar{\gamma}_{t_{m-1}}^{j,N}\brac{I_m}}\leq\sum_{m=1}^M\sum_{j=1}^K\brac{\bar{\beta}_{t_{m}}^{js,N}\brac{I_m}-\bar{\beta}_{t_{m-1}}^{js,N}\brac{I_m}}+\frac{MK}{N}.
\end{align*}
To bound it further, we prove the following. Fix $t_2>t_1$, then for any interval $A=\left[t_3,t_4\right)\subset\closedbrac{0,T}$ such that $t_2\geq t_4$:
\begin{align*}
    \sum_{i=1}^K\brac{\bar{\beta}_{t_2}^{is,N}\brac{A}-\bar{\beta}_{t_1}^{is,N}\brac{A}}\leq\sum_{n=0}^{\left\lfloor t_4/\epsilon\right\rfloor}\brac{\sum_{i=1}^K\brac{\bar{\alpha}_{t_2}^{i,N}\brac{A-n\epsilon}-\bar{\alpha}_{t_1}^{i,N}\brac{A-n\epsilon}}+\sum_{i=1}^K\bar{\xi}_{t_1}^{i,N}\brac{A-n\epsilon}+\frac{K}{N}}
\end{align*}
This will be shown by induction over $t_4<n\epsilon$. From \eqref{eq:hard Jackson dynamics} and the monotonicity of $t\mapsto\bar{\beta}_t^{ir,N}\brac{A}$, and $\bar{\xi}_{t_2}^{i,N}\left[0,t_4\right)=0$:
\begin{align*}
    \bar{\beta}_{t_2}^{is,N}\brac{A}-\bar{\beta}_{t_1}^{is,N}\brac{A}&\leq \bar{\alpha}_{t_2}^{i,N}\brac{A}-\bar{\alpha}_{t_1}^{i,N}\brac{A}+\sum_{j=1}^K\brac{\bar{\gamma}_{t_2}^{ji,N}\brac{A}-\bar{\gamma}_{t_1}^{ji,N}\brac{A}}+\bar{\xi}_{t_1}^{i,N}\brac{A},
\end{align*}
and by summing over all servers
\begin{align*}
    \sum_{i=1}^K\brac{\bar{\beta}_{t_2}^{is,N}\brac{A}-\bar{\beta}_{t_1}^{is,N}\brac{A}}&\leq \sum_{i=1}^K\brac{\bar{\alpha}_{t_2}^{i,N}\brac{A}-\bar{\alpha}_{t_1}^{i,N}\brac{A}+\bar{\xi}_{t_1}^{i,N}\brac{A}}\\
    &\quad+\sum_{j=1}^K\brac{\bar{\beta}_{t_2}^{js,N}\brac{A-\epsilon}-\bar{\beta}_{t_1}^{js,N}\brac{A-\epsilon}}+\frac{K}{N}.
\end{align*}
If $t_4<\epsilon$ then
\begin{align*}
    \bar{\beta}_{t_2}^{is,N}\brac{A}-\bar{\beta}_{t_1}^{is,N}\brac{A}\leq \bar{\alpha}_{t_2}^{i,N}\brac{A}-\bar{\alpha}_{t_1}^{i,N}\brac{A}+\bar{\xi}_{t_1}^{i,N}\brac{A},
\end{align*}
and the statement is true by summing over all servers. Assume that the statement is true for $t_4\leq n\epsilon$. If we take some $t_4\leq\brac{n+1}\epsilon$ and use our induction assumption we get:
\begin{align*}
    \sum_{i=1}^K\brac{\bar{\beta}_{t_2}^{is,N}\brac{A}-\bar{\beta}_{t_1}^{is,N}\brac{A}}&\leq \sum_{i=1}^K\brac{\bar{\alpha}_{t_2}^{i,N}\brac{A}-\bar{\alpha}_{t_1}^{i,N}\brac{A}+\bar{\xi}_{t_1}^{i,N}\brac{A}}\\
    &\quad+\sum_{n=1}^{\left\lfloor t_4/\epsilon\right\rfloor}\sum_{i=1}^K\brac{\bar{\alpha}_{t_2}^{i,N}\brac{A-n\epsilon}-\bar{\alpha}_{t_1}^{i,N}\brac{A-n\epsilon}+\bar{\xi}_{t_1}^{i,N}\brac{A-n\epsilon}+\frac{1}{N}}+\frac{K}{N}\\
    &=\sum_{n=0}^{\left\lfloor t_4/\epsilon\right\rfloor}\sum_{i=1}^K\brac{\bar{\alpha}_{t_2}^{i,N}\brac{A-n\epsilon}-\bar{\alpha}_{t_1}^{i,N}\brac{A-n\epsilon}+\bar{\xi}_{t_1}^{i,N}\brac{A-n\epsilon}+\frac{1}{N}}
\end{align*}
We need this result for $t_1=t_3=t_{m-1}$ and $t_2=t_4=t_m$; if $M$ is big enough then $\bar{\xi}_{t_{m-1}}^{i,N}\brac{A-n\epsilon}=0$ for $n\geq 1$, and then:
\begin{align*}
    \sum_{i=1}^K\brac{\bar{\beta}_{t_2}^{is,N}\brac{A}-\bar{\beta}_{t_1}^{is,N}\brac{A}}\leq \sum_{n=0}^{\left\lfloor t_4/\epsilon\right\rfloor}\sum_{i=1}^K\brac{\bar{\alpha}_{t_2}^{i,N}\brac{A-n\epsilon}-\bar{\alpha}_{t_1}^{i,N}\brac{A-n\epsilon}+\frac{1}{N}}+\sum_{i=1}^K\bar{\xi}_{t_1}^{i,N}\brac{A}.
\end{align*}
It follows that 
\begin{align*}
    G_{N,M}^i\leq\sum_{m=1}^M\frac{KT}{\epsilon}\max_{i,n}\brac{\bar{\alpha}_{t_m}^{i,N}\brac{I_m-n\epsilon}-\bar{\alpha}_{t_{m-1}}^{i,N}\brac{I_m-n\epsilon}}+K\sum_{m=1}^M\bar{\xi}_{t_{m-1}}^{i,N}\brac{I_m}+\brac{\frac{T}{\epsilon}+1}\frac{MK}{N}.
\end{align*}
The terms on the RHS vanish as $C_{N,M}^i$ and $D_{N,M}^i$ above.
\end{proof}

\paragraph{Acknowledgement.}
This research was supported in part by the ISF (grants 1184/16 and 1035/20). 


\footnotesize

\bibliographystyle{is-abbrv}

\bibliography{EDF}

\end{document}